\documentclass[11pt,leqno]{amsart}

\usepackage[utf8]{inputenc}
\usepackage[T1]{fontenc}
\usepackage{t1enc}
\usepackage{amsfonts,amsmath,amssymb,amsthm,color,amsbsy,mathtools}
\usepackage[margin=1in]{geometry}
\usepackage[colorlinks=true, pdfstartview=FitV, linkcolor=blue,citecolor=blue, urlcolor=blue]{hyperref}
\usepackage[abbrev,lite,nobysame]{amsrefs}
\usepackage{enumerate}
\usepackage{mathptmx}
\usepackage[all]{xy}
\usepackage{eucal}
\usepackage{graphicx}

\newtheorem{theorem}{Theorem}[section]
\newtheorem{lemma}[theorem]{Lemma}
\newtheorem{proposition}[theorem]{Proposition}
\newtheorem{corollary}[theorem]{Corollary}
\newtheorem{remark}[theorem]{Remark}
\newtheorem{example}[theorem]{Example}
\newtheorem{definition}[theorem]{Definition}
\newtheorem{assumption}[theorem]{Assumption}
\numberwithin{equation}{section}

\renewcommand{\geq}{\geqslant}
\renewcommand{\leq}{\leqslant}
\renewcommand{\tilde}{\widetilde}

\renewcommand{\H}{\mathbf{H}}
\renewcommand{\part}{\partial}
\renewcommand{\deg}{\mathrm{deg}}
\renewcommand{\Re}{\mathrm{Re}\,}
\renewcommand{\d}{\,\mathrm{d}}
\renewcommand{\le}{\left (}

\DeclareMathOperator*{\esssup}{ess\,sup} 
\DeclareMathOperator*{\essinf}{ess\,inf}
\DeclareMathOperator*{\argmin}{arg\,min}

\newcommand{\closure}[2][3]{%
  {}\mkern#1mu\overline{\mkern-#1mu#2}}
\newcommand{\cl}{\closure[0]}
\newcommand{\be}{\begin{equation}}
\newcommand{\ee}{\end{equation}}
\newcommand{\Fix}{\mathrm{Fix}}
\newcommand{\Coin}{\mathrm{Coin}}
\newcommand{\Lin}{\mathcal{L}}
\newcommand{\A}{\mathbf{A}}

\newcommand{\F}{\mathbf{F}}
\newcommand{\K}{\mathbf{K}}
\newcommand{\Q}{\mathbb{Q}}
\newcommand{\R}{\mathbb{R}}
\newcommand{\C}{\mathbb{C}}
\newcommand{\N}{\mathbb{N}}
\newcommand{\Z}{\mathbb{Z}}

\newcommand{\Gr}{\mathrm{Gr}}
\newcommand{\id}{\mathrm{id}}

\newcommand{\Ind}{\mathrm{Ind}}
\newcommand{\la}{\langle}
\newcommand{\ra}{\rangle}
\newcommand{\vp}{\varphi}

\newcommand{\f}{\varphi}
\newcommand{\eps}{\varepsilon}

\newcommand{\supp}{\mathrm{supp}\,}
\newcommand{\pr}{\right )}

\newcommand{\multi}{\multimap}

\author[W. Kryszewski]{Wojciech Kryszewski}
\address{Institute of Mathematics, Lodz University of Technology, Lodz, Poland}
\email{wojciech.kryszewski@p.lodz.pl}
\author[J. Siemianowski]{Jakub Siemianowski}
\address{Institute of Mathematics, Polish Academy of Sciences, Śniadeckich 8, Warsaw, Poland}
\email{jsiem@mat.umk.pl}

\title[Constrained Elliptic Systems]
{Constrained Semilinear Elliptic Systems on $\R^N$}

\date{\today}


\begin{document}
\pagestyle{myheadings}
\baselineskip15pt
\begin{abstract}
The existence of solutions $u$ in $H^1(\R^N,\R^M)\cap H^2_{loc}(\R^N,\R^M)$ of the  coupled semilinear system of second order elliptic partial differential equations on $\R^N$ of the form
\[
\mathcal{P}[u] = f(x,u,\part u), \quad x\in \R^N,
\]
under pointwise constraints is considered. The problem is studied via the constructed suitable topological invariant, the so-called constrained topological degree, which allows to get the existence of solutions of abstract problems considered as $L^2$-realizations of the approximating sequence  of systems obtained by the truncation of the initial system to bounded subdomains.  The key step of the proof consists in showing  the relative $H^1$-compactness of the sequence of solutions to the truncated systems by the use of the so-called tail estimates. The constructions rely on the semigroup approach combined with topological methods, as well as invariance/viability techniques.
\end{abstract}
\maketitle

\section{Introduction}
In the paper we discuss the existence of solutions $u\in H^1(\R^N,\R^M)$ to a strongly coupled system of semilinear  elliptic partial differential equations
$$\mathcal P[u]=f(x,u,\part u),\;x\in\R^N,\leqno (1)$$
where $\part u$ is the Jacobian matrix of the unknown function $u:\R^N\to\R^M$. Here, $f:\R^N\times\R^M\times\R^{M\times N}\to\R^M$ is a vector-valued function and $\mathcal P$ is a linear second-order elliptic partial differential operator of the form
${\mathcal P}[u]=-\sum_{i,j=1}^N\part_i(A^{ij}\part_ju)+\sum_{i=1}^NB^i\part_iu+Cu$ with the coefficients
 $A^{ij}$, $B^i$ and $C$ being  functions from $\R^N$ into $\R^{M\times M}$, the space of $M\times M$ matrices. In applications such systems describe steady  states of evolution processes involving $M$ unknown species or quantities subject to diffusion $\mathcal P$ and the forcing term $f$ including the advection or drift effects. The form of $\mathcal P$ allows interactions between species occur on the diffusion level, too (see e.g. \cite{Chueh} and more recent \cite{Marc}). We look for solutions $u$ to (1) satisfying {\em pointwise constraints} of the form
$u(x)\in K(x)$ for a.a. $x\in\R^N$, where the set $K(x)\subset \R^M$ is closed and convex.\\
\indent Constrained problems of this form arise  in various applications, where natural bounds for the unknown quantities are present.  For instance, experimentally obtained lower and upper  threshold values $\sigma_k,\tau_k$ (depending on $x$) are often given and solutions $u=(u_1,\ldots,u_M)$ satisfying $\sigma_k(x) \leq u_k(x)\leq\tau_k(x)$ a.e.  for  $1\leq k\leq M$ are sought-after. As it also seems,  constrained solutions appear sometimes {\em a posteriori} as a by-product consequence of assumptions relaxing the standard growth condition.   This is the case when the method of sub- and supersolutions is applied, e.g., if  $\Omega\subset\R^N$ is bounded, $f:\Omega\times\R\to\R$ is sufficiently regular and there are constants $\alpha\leq 0\leq\beta$ such that
\[
\tag{M} f(x,\alpha)\geq 0, \; f(x,\beta)\leq 0\quad \text{ for }x\in\Omega,
\]
then there is $u\in H^1_0(\Omega)$ such that $-\Delta u(x)=f(x,u(x))$ and $\alpha\leq u(x)\leq \beta$ for a.a. $x\in\Omega$ (see also \cite{S},  \cite{McKenna} and \cite{Miti}). The technique relying on the so-called ,,invariant regions'' or ,,rectangles'' is often used to show the global existence in time of parabolic evolution problems as well as the related stationary problems. The mentioned method is related to the method of upper and lower solutions and akin to arguments started apparently almost 90 years ago by Max M\"uller \cite{Muller}, \cite{Muller-book}; conditions like  (M) (and their relatives) are sometimes called the {\em M\"uller conditions} (see \cite{Walter-survey}, \cite[\S 12.IV, \S32.V]{Walter-book}).\\
\indent There is a vast bibliography on invariant sets (or the so-called {\em viablity} theory -- the term used mainly by French and Romanian mathematicians) for nonlinear parabolic systems of the form $u_t+\mathcal P[u]=f(u)$ on a bounded $\Omega$ and their elliptic counterparts (2). In the elliptic case we prefer to talk about {\em constraints} since no evolution takes place. The reader is referred especially to the seminal work of Amann \cite{Amann}, the books \cite{Martin} of Martin and \cite[Theorem 14.7, p. 200]{Smoller} of Smoller. These and other authors deal mainly with a bounded $\Omega$ and a closed convex bounded set $K(\cdot)=K$ (i.e., $K(\cdot)$ is {\em independent} of $x\in\Omega$) containing the origin.  For instance the case $A^{ij}=a_{ij}I$, $B^i=b_iI$ and $C=cI$  for $i,j=1,\ldots,N$, where $a_{ij}, b_i, c\in\R$ and $I$ stands for the $M\times M$ identity matrix, and $K$ is an arbitrary convex set is studied in \cite{Wein, Beber, Redheffer, Redlinger}. The case $A^{ij}=\delta_{ij}D$, where $\delta_{ij}$ is the Kronecker delta, $D=\mathrm{diag}(d_1,\ldots,d_M)$, $B^i=C=0$, and $K=\{u\in\R^M\mid G(u)\leq c\}$, where $G:\R^M\to\R$ is a quasi-convex function was studied, e.g., in \cite{Alikakos, Plum}. A different situation was considered in \cite{Chueh} and \cite{Smoller}; the authors deal with $A^{ij}=\delta_{ij}D$, where $D$ is an $M\times M$ not necessarily diagonal matrix and $B^i=C=0$. In \cite{Amann, Kuiper1, Kuiper2}, matrices $A^{ij}$, $B^i$ and $C$ are diagonal and $K$ is a rectangle (or the Cartesian product of closed convex and bounded sets), while in \cite{Kuiper1, Kuiper2} $K$ is no longer constant but $K(x)=\{u\in\R^M\mid \sigma_k(x)\leq u_k\leq \tau_k(x), k=1,\ldots,M\}$, where $\sigma,\tau:\Omega\to\R^M$. In all of theses papers the forcing term $f(x,u,\cdot)$ is, roughly speaking, assumed to be directed  inward the set $K$ for $x\in\Omega$ and $u\in\part K$.\\
\indent As it appears, the form of coefficients $A^{ij}$, $B^i$ and $C$, $i,j=1,\ldots,N$,  along with the shape of the constraining sets implicitly imply their invariance  with respect to homogeneous problem $u_t+\mathcal P[u]=0$. This actually follows from an appropriately  used version of the maximum principle, since the componentwise or the norm maximum principle for parabolic or elliptic systems can be interpreted as the invariance of an orthant or a ball, respectively.\\
\indent The above-mentioned ,,inwardness'' condition is subsumed by the {\em tangency} expressed in the language of convex analysis. This has already been observed in \cite{Martin} and thoroughly discussed in \cite{Carja}, where a part of the theory development (not reflected elsewhere) is presented with a detailed bibliography. A variety of the so-called Nagumo, Pavel's and other conditions are formulated and studied there. \\
\indent In the present paper we deal with a general elliptic operator $\mathcal P$ and a general convex constraint $K(x)\subset\R^M$, $x\in\Omega\subseteq \R^N$ (see Sections 1.2.2 and 4.3); conditions implying the invariance of constraints under the flow induced by $\mathcal P$ are imposed in assumption \ref{additional ass} and studied in Subsections \ref{INV} and \ref{co}.  The nonlinearity $f$ in (1) satisfies relaxed  regularity conditions and  is subject to the tangency condition (see assumption \ref{nonlinear}).\\
\indent Due to the lack of compactness general fixed point methods are not readily applicable to solve (1) directly. In this case, we establish the existence of a  solution by solving a sequence of the approximating Dirichlet boundary value problems truncated to $\Omega_n$, $n\in\N$, where $\Omega_n=B(0,R_n)$ is the open ball with $R_n\to\infty$. Roughly speaking, a solution is obtained then as a limit of those approximate solutions. For that reason along with (1) we study problems of the form
\vspace{0mm}
$${\mathcal P}[u]=f(x,u,\part u),\;\;u|_{\part\Omega}=0,\; u(x)\in K(x),\; x\in\Omega,\leqno(2)$$
where $\Omega$ is an open bounded subset of $\R^N$.

\indent For a bounded $\Omega$ our results follow from more general abstract results concerning a problem
$$\A(u)=\F(u),\;\; u\in\K,\leqno (3)$$
where $\K$ is a closed convex subset of a Banach space $X$,  $\A:D(\A)\to X$ is a sectorial operator in $X$, $\F:U\to X$ is a continuous map defined on an open subset $U$ of $\K^\alpha:=\K\cap X^\alpha$ with $X^\alpha$ being the fractional space corresponding to $\A$, $\alpha\in [0,1)$. Such a framework  creates a convenient abstract scheme for (2), where $\A$ corresponds to the $L^p$-realization of $\mathcal P$, $\F$ is the superposition operator generated by $f$ and $\K=\{u\in L^p\mid u(x)\in K(x)\;\text{a.e}\}$. Even though $\A$ is resolvent compact,  the existence of solutions to $(3)$  can hardly be obtained via a direct use of the Leray--Schauder theory. There are two difficulties: firstly, in general $\F$ is only locally bounded and $\K^\alpha$ is unbounded, secondly, $\K$ and $\K^\alpha$ have empty interiors. Moreover, neither $\A$ nor $\F$ maps $\K$ (or $\K^\alpha$) into $\K$ so the Leray--Schauder fixed point index of maps on absolute neighborhood retracts  cannot be employed, too.
We propose instead to apply a topological invariant responsible for the existence of solutions of $(3)$. The construction of the so-called {\em constrained topological degree} is the next subject of the paper, see Subsection \ref{constrained degree}.
It relies on the assumption of  the so-called {\em tangency} of $\F$ and the {\em invariance} of $\K$  with respect to the semigroup generated by $-\A$.\\
\indent A similar abstract scheme has also been considered in \cite{KryszewskiSiemianowski}, where $\K$ and $\F(\K^\alpha)$ are assumed to be bounded. Roughly speaking, the existence results obtained in \cite{KryszewskiSiemianowski} were based on  approximations similar to \eqref{defdeg} but based on  the Schauder fixed-point theorem rather. However, the boundedness assumptions do restrict the applications significantly.

The paper is organized as follows. We recall some standard notation and introduce the problem below. In the case of a bounded  domain, the existence of constrained solutions to the  problem is established in Section 2 (see Theorem \ref{thm:3.1}). The case $\Omega=\R^N$ is studied and the main result (Theorem \ref{main th}) is proved in the third section. In the last section we recall some relevant information about sectorial operators, we discuss the invariance of closed convex sets (see, e.g., Propositions \ref{invariance 1}, \ref{invariance 2}) , we study some examples of the sets of constraints and, in Subsection \ref{constrained degree}, we provide the construction of the constrained topological  degree and its properties.
\subsection{Preliminaries}
Throughout the paper $\R^N$ denotes the standard $N$-dimensional real Euclidean space and $\R^{M\times N}$ the space of  $M\times N$ real matrices. The norm in $\R^N$ or $\R^{M\times N}$ is denoted by $|\cdot|$; the scalar product in $\R^M$ (resp. the Frobenius product in $\R^{M\times N}$) is denoted by $\la\cdot,\cdot\ra$. For example if $\xi,\zeta\in\R^{M\times N}$, then $\la\xi,\zeta\ra:=\sum_{k=1}^M\sum_{l=1}^N\xi_{kl}\zeta_{kl}$ and $|\xi|^2:=\la\xi,\xi\ra$. By $\prescript\intercal{}\!A$ we denote the transpose of a matrix $A$.\\
\indent   By $\Omega\subseteq\R^N$ we usually denote a  bounded domain of class $C^2$. Given a locally integrable map $u=(u_1,\ldots,u_M)$ from $\Omega$ to $\R^M$, $\part u$ is the distributional Jacobian matrix of $u$, i.e., $\part u:=[\part_iu_k(\cdot)]_{i=1,\ldots,N}^{k=1,\ldots,M}\in\R^{M\times N}$, where
$\part_iu_k:=\frac{\part}{\part x_i}u_k$ is the $i$-th partial derivative understood in the sense of distributions; $\part_iu:=[\part_iu_k]_{k=1}^M$ is the $i$-th column of $\part u$. Given a multi-index $\alpha\in\Z_+^N$, $\part ^\alpha:=\part_1^{\alpha_1}\ldots\part^{\alpha_N}_N$ and $|\alpha|=\sum_{i=1}^N\alpha_i$. \\
\indent  $L^p(\Omega,\R^M)$, $1\leq p<\infty$, denotes the space of measurable functions $u:\Omega\to\R^M$ such that $|u|^p$ is Lebesgue integrable with the standard norm $\|u\|_{L^p(\Omega,\R^M)}:=(\int_\Omega |u(x)|^p\,\d x)^{1/p}$; $L^\infty(\Omega,\R^M)$ is the space of measurable functions $u:\Omega\to\R^M$ with $\|u\|_{L^\infty(\Omega,\R^M)}:=\esssup_{x\in \Omega}|u(x)|<\infty$.\\
\indent $W^{k,p}(\Omega,\R^M)$ (resp. $W^{k,p}_0(\Omega,\R^M)$), $k\in \N$, $1\leq p\leq\infty$,  stands for the Sobolev space of functions $u: \Omega\to \R^M$ having weak partial derivatives up to order $k$ in  $L^p(\Omega,\R^M)$ (resp. and vanish at the boundary, in the sense of the trace) with the standard norm
$$\|u \|_{W^{k,p}(\Omega,\R^M)}: =
\left(\sum_{|\alpha| \leq k}\|\part^\alpha u\|^p_{L^p}\right)^{1/p}\; \text{if}\; p<\infty, \; \text{ and}\; \|u \|_{W^{k,\infty}(\Omega,\R^M)}=\sum_{|\alpha|\leq k}\esssup_{x\in \Omega}|\part^\alpha u(x)|.$$
We write $H^k$ (resp. $H^k_0$) instead of $W^{k,2}$ (resp. $W^{k,2}_0$); clearly $H^1_0(\R^N,\R^M)=H^1(\R^N,\R^M)$. By $H^{-1}(\Omega,\R^M)$ we denote the dual of $H^1_0(\Omega,\R^M)$.\\
\indent  It is  convenient to consider the seminorms $|\cdot|_{j,p,\Omega}$, where  $0\leq j \leq k$, in $W^{k,p}(\Omega,\R^M)$ putting
\[
|u|_{j,p,\Omega} := \le \sum_{|\alpha| = j}\|\part^\alpha u\|_{L^p}^p\pr ^{1/p}.\]
 If $M=1$, then symbol $\R^M$ will be suppressed form the notation concerning spaces.
\subsection{The problem and general assumptions} We  study the existence of solutions to the following system of elliptic equations
\be\label{main problem}\mathcal P[u]=f(x,u,\part u),\;\;x\in\Omega,\ee
where either $\Omega=\R^N$ or $\Omega$ is a bounded  domain in $\R^N$ and $u=(u_1,\ldots,u_M):\Omega\to\R^M$. If $\Omega$ is bounded, then \eqref{main problem} is studied {\em subject to the Dirichlet boundary condition} $u|_{\part\Omega}=0$. In both cases we are interested in the existence of solutions $u$ to \eqref{main problem} such that
\be\label{cons}u(x)\in K(x)\;\;\text{for a.a}\;\;x\in\Omega,\ee
where $K(x)\subset\R^M$  for each $x\in\Omega$.
General assumptions  \ref{constraint} concerning $K(\cdot)$, \ref{nonlinear} concerning $f$ and \ref{differential operator} concerning $\mathcal P$ are presented below. These assumptions will be discussed, illustrated and appropriately complemented in the following subsections.
\begin{assumption}\label{constraint} {\em We assume that:
\begin{enumerate}
\item for each $x\in\Omega$, $K(x)$ is nonempty {\em closed and convex} subset of $\R^M$;
\item for any open  $U\subset\R^M$, the set $\{x\in\Omega\mid K(x)\cap U\neq\emptyset\}$ is (Lebesgue) {\em measurable} (\footnote{This means that the set-valued map $\Omega\ni x\mapsto K(x)\subset\R^M$ is {\em measurable}; see \cite{Aubin} for details.});
\item there is $m\in L^2(\Omega)$ such that $\sup_{u\in K(x)}|u|\leq m(x)$  for a.a. $x\in\Omega.$
    \end{enumerate}
}\end{assumption}
For $x\in\Omega$,  let $r(x,\cdot)$ be the metric projection of $\R^M$ onto $K(x)$:
\be\label{retra}|u-r(x,u)|=d(u,K(x)):=\inf_{w\in K(x)}|u-w|,\; u\in\R^M.\ee
Then $r:\Omega\times\R^M\to\R^M$ is well-defined; for any $x\in\Omega$, $r(x,\cdot)$ is nonexpansive and, hence,
$$|r(x,u)|\leq d(0,K(x))+|u|\leq m(x)+|u|$$ for all $u\in\R^M$ and a.a. $x\in\Omega$. By (2) and \cite[Cor. 8.2.13]{Aubin}, for any $u\in\R^M$, $r(\cdot,u)$ is measurable. Therefore, in view of the Krasnosel'skii theorem on superpositions, the Nemytski operator $u\mapsto r(\cdot,u(\cdot))$ maps $L^2(\Omega,\R^M)$ continuously into itself. Given $u\in L^2(\Omega,\R^M)$, $r(\cdot,u(\cdot))$ is an $L^2$-selection of $K(\cdot)$, i.e., $r(x,u(x))\in K(x)$ for a.a. $x\in\Omega$.\\
\indent In Subsection \ref{co} we present some examples of $K(\cdot)$ satisfying the above conditions.
\begin{assumption}\label{nonlinear} {\em Let $f:\Gr(K)\times\R^{M\times N}\to\R^M$, where $\Gr(K):=\{(x,u)\in\Omega\times\R^M\mid u\in K(x)\}$ is the {\em graph} of $K(\cdot)$. We assume that
\begin{enumerate}
\item $f$ is a {\em Carath\'eodory map}, i.e., for a.a. $x\in\Omega$, $f(x,\cdot):K(x)\times\R^{M\times N}\to\R^M$ is continuous and for all $u\in\R^M$, $\xi\in\R^{M\times N}$ the map $f(\cdot,u,\xi)$ defined on $\{x\in\Omega\mid u\in K(x)\}$ is measurable (\footnote{Observe that in view of assumption \ref{constraint} (2) the set $\{x\in\Omega\mid u\in K(x)\}$ is measurable for any $u\in\R^M$, so the condition makes sense.});
\item there are $\beta\in L^2(\Omega)$, $c>0$, $1\leq s<\frac{N+4}{N}$ and $1\leq q< \frac{N+4}{N+2}$ such that
\be\label{growth}
|f(x,u,\xi)| \leq \beta(x) + c(|u|^s + |\xi|^q),\quad x\in\Omega,\;u\in K(x),\;\xi\in\R^{M\times N};
\ee
\item $f(x,\cdot,\cdot)$ is {\em tangent} to $K(x)$, i.e., for a.a. $x\in \Omega$ and for all $u\in K(x)$, $\xi\in\R^{M\times N}$
\be\label{tangency}f(x,u,\xi)\in T_{K(x)}(u),\ee
where $T_{K(x)}(u)$ is the tangent cone to the set $K(x)$ at $u$ (see Section \ref{sto} for the definition of the tangent cone).
\end{enumerate}
}\end{assumption}
\noindent Condition (3) means that for all $x\in\Omega$, $u\in K(x)$ the forcing vector field driven by $f(x,u,\cdot)$ with its tain at $u$ is directed {\em inward} the set $K(x)$. It will be illustrated and discussed  below (see Sections 4.2 and 4.3).  Note that the growth condition (2) is considered on $\Gr(K)$ only. If, for instance $\Omega$ is bounded, $K(\cdot)=K$, where $K\subset \R^M$ is convex and compact,  $f$ is continuous and depends only on $u$, then no growth condition is necessary at all.
\begin{assumption}\label{differential operator}{\em By $\mathcal P$ we denote a linear differential operator in the divergence form
\be\label{diffop}\mathcal P[u]:=-\sum_{i,j=1}^N\part_i\left(A^{ij}(x)\part_ju\right)+\sum_{i=1}^NB^i(x)\part_iu+
C(x)u,\ee
where the coefficients $A^{ij}=[A^{ij}_{kl}]_{k,l=1}^M$, $B^i=[B^i_{kl}]_{k,l=1}^M$ and $C=[C_{kl}]_{k,l=1}^M$ are functions from $\Omega$ into $\R^{M\times M}$ such that   $A^{ij}_{kl}\in C^{0,1}(\Omega)\cap L^\infty(\Omega)$, $B^i_{kl}, C_{kl}\in L^\infty(\Omega)$ for $1\leq i,j\leq N$, $1\leq k,l\leq M$.
}\end{assumption}

The operator $\mathcal P$ acts on a column vector-valued function $u=\prescript\intercal{}(u_1,\ldots,u_M)$ in the sense of distributions returning the vector-valued function $\mathcal P[u]:\Omega\to\R^M$ with components
$$\mathcal P[u]_k=-\sum_{i,j=1}^N\sum_{l=1}^M\part_i\left(A^{ij}_{kl}\part_ju_l\right)
+\sum_{i=1}^N\sum_{l=1}^MB^i_{kl}\part_iu_l+\sum_{l=1}^MC_{kl}u_l,\;\;k=1,\ldots,M.$$
\indent With  $\mathcal P$ we associate a bilinear form $\mathcal B$ on $H^1_0(\Omega,\R^M)$ given by
\begin{align}\label{form}\mathcal B[u,v]:&=\int_\Omega\left(\sum_{i,j=1}^N\la A^{ij}\part_ju,\part_iv\ra+\sum_{i=1}^N\la B^i\part_iu,v\ra+\la Cu,v\ra\right)\d x\\
\nonumber &=\int_\Omega\left(\sum_{i,j=1}^N\sum_{k,l=1}^M
A_{kl}^{ij}\part_ju_l\part_iv_k+\sum_{i=1}^N\sum_{k,l=1}^M
B^i_{kl}\part_iu_l v_k+\sum_{k,l=1}^MC_{kl}u_l v_k\right)\d x\end{align}
for $u,v\in H^1_0(\Omega,\R^M)$. Clearly $\mathcal B$ is bounded
$$|\mathcal B[u,v]|\leq c\|u\|_{H^1(\Omega,\R^M)}\|v\|_{H^1(\Omega,\R^M)} (\footnote{We adopt the usual convention to denote by $c$ a {\em general} positive constant, that may vary from line to line. Peculiar dependence on parameters will be emphasized in parantheses when needed.}),$$
where the constant depends on the $L^\infty(\Omega)$-norms of $A^{ij}_{kl}$, $B^i_{kl}$ and $C_{kl}$. Observe that for $u\in H^2_{loc}(\Omega,\R^M)$ the expression $\mathcal P[u](x)$ makes sense for a.a. $x\in\Omega$  and ${\mathcal P}[u]\in L^2_{loc}(\Omega,\R^M)$; if
$u\in H^2(\Omega,\R^M)\cap H_0^1(\Omega,\R^M)$, then, by the Green identity, ${\mathcal B}[u,v]=\la {\mathcal P}[u],v\ra_{L^2}$ for any $v\in H^1_0(\Omega,\R^M)$.

\begin{definition}\label{solution}{\em We say that $u\in H^1_0(\Omega,\R^M)$ is a {\em weak solution} of \eqref{main problem} if for any $v\in H_0^1(\Omega,\R^M)$
$$\mathcal B[u,v]=\int_\Omega \la f(x,u,\part u),v\ra\d x.$$
If $u\in H^1_0(\Omega,\R^M)\cap H^2(\omega;\R^M)$ for any bounded open $\omega\subset\Omega$  and $\mathcal P[u]=f(\cdot,u,\part u)$ a.e. in $\Omega$, then $u$ is a called a {\em strong} solution.
}\end{definition}

\begin{remark}\label{about sol}{\em (i) Note that the definition of a weak solution makes sense since, for any $u,v\in H^1_0(\Omega,\R^M)$, $\la f(\cdot,u.\part u),v\ra\in L^1(\Omega)$. Indeed, in view of $(f_2)$,
\be\label{tosol}\int_\Omega\la f(x,u,\part u),v\ra\d x\leq \int_\Omega|f(x,u,\part u)||v|\d x\leq c(I_1+I_2+I_3),\ee
where
$$I_1:=\int_\Omega\beta|v|\d x\leq\|\beta\|_{L^2(\Omega)}\|v\|_{L^2(\Omega,\R^M)},
\;\;I_2:=
\int_\Omega|u|^s|v|\d x,\;\;  I_3:=\int_\Omega|\part u|^q|v|\d x.$$
In order to estimate $I_3$, let $p:=2q^{-1}$. Then $p>1$ and let $p':=p(p-1)^{-1}=2(2-q)^{-1}$. It is easy to see that $p'\in [2,2^\ast)$ (\footnote{Given $1\leq p<N$, $p^\ast$ stands for the Sobolev critical exponent, i.e., $\frac{1}{p^\ast}=\frac{1}{p}-\frac{1}{N}$, in particular $2^\ast=\frac{2N}{N-2}$.}). Since $H^1(\Omega,\R^M)\subset L^{p'}(\Omega,\R^M)$, by the H\"older inequality we have
\begin{align*}I_3&\leq \left(\int_\Omega|\part u|^{pq}\d x\right)^{1/p}\left(\int_\Omega|v|^{p'}\d x\right)^{1/p'}\leq c\|\part u\|^q_{L^2(\Omega,\R^M)}\|v\|_{H^1(\Omega,\R^M)}\\ &\leq c\|u\|_{H^1(\Omega,\R^M)}^q\|v\|_{H^1(\Omega,\R^M)}.\end{align*}
Now suppose that $1\leq s<\frac{N+2}{N}$ and let $t:=2s^{-1}>1$ and $t':=t(t-1)^{-1}$. It is easy to see that $t'\in [2,2^\ast)$. Hence $v\in H^1(\Omega,\R^M)\subset L^{t'}(\Omega,\R^M)$ and, by the H\"older inequality
$$I_2\leq c\|u\|^s_{L^2(\Omega,\R^M)}\|v\|_{H^1(\Omega,\R^M)}.$$
For $\frac{N+2}{N}\leq s<\frac{N+4}{N}$ we let $t:=2^\ast(2^\ast-1)^{-1}$ and $t':=2^\ast$. Then $ts\in [2,2^\ast)$ so, by the H\"older inequality
$$I_2\leq\left(\int_\Omega|u|^{ts}\d x\right)^{1/t}\left(\int_\Omega|v|^{t'}\right)^{1/t'}\leq c\|u\|^s_{H^1(\Omega,\R^M)}\|v\|_{H^1(\Omega,\R^M)}.$$
\indent Summing up, \eqref{tosol} implies
that given $u\in H^1(\Omega,\R^M)$, a mapping  $H^1(\Omega,\R^M)\ni v\mapsto \ell(v):=\la f(\cdot,u,\part u),v\ra_{L^2}$ is a well defined linear functional and $|\ell(v)|\leq c\|v\|_{H^1(\Omega,\R^M)}$. Thus, $u\in H^1_0(\Omega,\R^M)$ is a weak solution if and only if
$$\mathcal B[u,\vp]=\int_\Omega \la f(x,u,\part u),\vp\ra\d x$$
for any $\vp\in C^\infty_0(\Omega,\R^M)$.\\
\indent (ii) The definition of a strong solution also makes sense since, arguing as above, if $u\in H^2_{loc}(\omega,\R^M)$ for some $\omega\subset\Omega$, then $f(\cdot,u,\part u)\in L^2(\omega,\R^M)$. Clearly strong solutions are weak. If a weak solution $u$ belongs to $H^2(\omega,\R^M)$ for any bounded open $\omega\subset\Omega$, then $u$ is a strong solution. In particular if $\Omega$ is bounded, then a weak solution $u\in H^1_0(\Omega,\R^M)$ is strong if and only if $u\in H^2(\Omega,\R^M)$.
}\end{remark}

\section{Bounded domain}\label{bdd}

In this section $\Omega$ is a {\em smooth bounded} domain in $\R^N$. We are going to establish the existence of strong solutions to \eqref{main problem}, \eqref{cons}, i.e.,
\be\label{eq:3.1}{\mathcal P}[u]=f(x,u,\part u),\;\; u|_{\part\Omega}=0,\ee
such that
\be\label{eq:3.2} u(x)\in K(x)\;\; \text{for a.a}\; x\in\Omega,\ee
where $K(\cdot)$, $f$ and $\mathcal P$ satisfy assumptions \ref{constraint}, \ref{nonlinear} and \ref{differential operator}, respectively.
\begin{assumption}\label{additional ass}{\em We make the following additional assumptions
\begin{enumerate}
\item the operator \eqref{diffop} is {\em elliptic} in the sense of the {\em Legendre-Hadamard condition}, i.e., there is {\em an ellipticity constant} $\theta>0$ such that for any $\zeta\in\R^N$ and $p\in\R^M$
    \be\label{ellipticity}\sum_{i,j=1}^N\la A^{ij}\zeta_jp,\zeta_ip\ra=\sum_{i,j=1}^N\sum_{k,l=1}^MA^{ij}_{kl}(x)
p_k p_l\zeta_i\zeta_j\geq\theta|p|^2|\zeta|^2.\ee
\item the graph $\Gr(K)$ is {\em viable} (or invariant) with respect to the {\em `diffusion' flow}; this means that, for fixed $T>0$ given $u_0\in L^2(\Omega,\R^M)$ such that $u_0(x)\in K(x)$ for a.a. $x\in\Omega$, the weak solution $u:[0,T]\times \Omega\to\R^M$ to the corresponding {\em linear parabolic} Cauchy problem
\be\label{parabolic}u_t+\mathcal P[u]=0,\;\;u(0,\cdot)=u_0,\;u|_{\part\Omega}=0\ee
stays in $\Gr(K)$, i.e., $u(t,x)\in K(x)$  for all $t> 0$ and a.a. $x\in\Omega$.
\end{enumerate}
}\end{assumption}
\noindent Recall that a function $u\in L^2(0,T;H^1_0(\Omega,\R^M))$ such that $u'\in L^2(0,T;H^{-1}(\Omega,\R^M))$  is a {\em weak solution} to \eqref{parabolic} if $u(0)=u_0$ and $[u'(t),v]+{\mathcal B}[u(t),v]=0$ for a.a. $t\in [0,T]$ and any $v\in H_0^1(\Omega,\R^M)$,
where $[\cdot,\cdot]$ is the duality pairing between $H^{-1}(\Omega,\R^M)$ and $H^1_0(\Omega,\R^M)$.
\begin{remark}\label{cons of additional}{\em (i) The flow invariance granted by the above condition (2) is a strong assumption.  In Corollary \ref{invariance 3} we shall show that it holds if and only if
$$r(\cdot,u(\cdot))\in H_0^1(\Omega,\R^M)\;\;\text{and}\;\; \mathcal B[r(\cdot,u(\cdot)),u-r(\cdot,u(\cdot))]\geq 0$$
for every $u\in H^1_0(\Omega,\R^M)$ (see \eqref{retra}). The flow invariance  will be illustrated and carefully discussed from different points of view below; see Subsection \ref{INV}, Propositions \ref{invariance 1}, \ref{invariance 2}, \ref{invev} and \ref{invev 1}  for details.\\
\indent (ii) The Legendre-Hadamard condition and its stronger form, the Legendre condition or {\em strong ellipticity}, is discussed in, e.g.,  \cite{Giaquinta} or \cite{McLean}. Let us recall  the following important consequence of ellipticity, the so-called G{\aa}rding inequality (see \cite[Theorem 3.42]{Giaquinta}, \cite[Theorem 4.6]{McLean} and \cite{Guesmia}). \\
\indent {\em If the operator $\mathcal P$ is elliptic, then the form $\mathcal B$ is {\em weakly coercive}, i.e., there are $\omega\in\R$ and $\alpha>0$ such that $\mathcal B[u,u]+\omega\|u\|_{L^2(\Omega,\R^M)}\geq\alpha\|u\|_{H^1(\Omega,\R^M)}$ for $u\in H_0^1(\Omega,\R^M).$}
}\end{remark}

In order to proceed we first establish {\em a priori} bounds for solutions to \eqref{eq:3.1}, \eqref{eq:3.2}. As in many other situations the existence of {\em a priori} bounds along with an appropriate continuation procedure will guarantee the existence.
\begin{lemma}\label{lem:3.1}
There is  $C_1>0$ (depending on $\Omega$, the ellipticity constant and the $C^{0,1}$-norms of  $A^{ij}$, $1\leq i,j\leq N$)  such that $\|u\|_{H^2(\Omega,\R^M)} \leq C_1$\, for every strong solution $u$ of \eqref{eq:3.1} satisfying \eqref{eq:3.2}.
\end{lemma}
\begin{proof}
Let $u\in H^1_0(\Omega,\R^M)\cap H^2(\Omega,\R^M)$ be a solution of \eqref{eq:3.1} satisfying \eqref{eq:3.2}. By $(K_3)$
\begin{equation}\label{eq:3.3}
\|u\|_{L^2(\Omega,\R^M)}\leq M_0:=\|m\|_{L^2}.
\end{equation}
Assumption $(f_2)$ shows that $f(\cdot,u,\part u)\in L^2(\Omega,\R^M)$. Indeed
\be\label{szac a}
\|f(\cdot,u,\part u)\|_{L^2(\Omega,\R^M)}\leq c(\|\beta\|_{L^2(\Omega)} + I_1+I_2)\ee
for some constant $c>0$, where
$$I_1:=\le\int_\Omega |u|^{2s}\d x\pr ^{1/2},\;\;I_2:=\le \int_\Omega|\part u|^{2q}\d x\pr^{1/2}.$$
\indent To estimate $I_1$ we consider two cases $3\leq N\leq 4$ and $N\geq 5$ separately. If $3\leq N \leq 4$, then
\[
2s<\frac{2N+8}{N}\leq 2^\ast,
\]
and the interpolation inequality with exponents $2s\in [2,2^\ast)$, $\theta_1\in [0,1)$
\begin{equation}\label{eq:3.6}
\frac{1}{2s}= \frac{1-\theta_1}{2} + \frac{\theta_1}{2^\ast}\iff \frac{s}{2}\theta_1 = \frac{N}{4}(s-1)
\end{equation}
shows
\[
I_1  = \|u\|_{L^{2s}(\Omega,\R^M)}^s\leq \|u\|_{L^2(\Omega,\R^M)}^{s(1-\theta_1)}\|u\|_{L^{2^\ast}(\Omega,\R^M)}^{s\theta_1}.
\]
By \eqref{eq:3.3} and by the continuity of the embedding $H^1(\Omega,\R^M)\hookrightarrow L^{2^\ast}(\Omega,\R^M)$
\[
I_1\leq c M_0^{s(1-\theta_1)}\|u\|_{H^1(\Omega,\R^M)}^{s\theta_1}.
\]
The Ehrling--Browder inequality (see \cite[Thm 4.17]{Adams}) implies that
\[
\|u\|_{H^1(\Omega,\R^M)}^{s\theta_1}\leq c \|u\|_{H^2(\Omega,\R^M)}^{\frac{s}{2}\theta_1}
\|u\|_{L^2(\Omega,\R^M)}^{\frac{s}{2}\theta_1}.
\]
Hence and again from \eqref{eq:3.3}
\be\label{szac d}
I_1\leq c\|u\|_{H^2(\Omega,\R^M)}^{\gamma_1},
\ee
where $\gamma_1:=\frac{s}{2}\theta_1<1$, since \eqref{eq:3.6} and $s<\frac{N+4}{N}$.\\
\indent If $N\geq 5$, then we use the interpolation inequality with exponents $2s\in [2, 2^{\ast\ast})$(\footnote{The symbol $2^{\ast\ast}$ stands for $\le 2^\ast\pr ^\ast$, i.e., $2^{\ast\ast} = \frac{2N}{N-4}$.}), $\tilde\theta_1\in [0,1)$ satisfying
\begin{equation}\label{eq:3.7}
\frac{1}{2s}=\frac{1-\tilde\theta_1}{2}+\frac{\tilde\theta_1}{2^{\ast\ast}}\iff s\tilde\theta_1 = \frac{N}{4}(s-1)
\end{equation}
to get that
\[
I_1=\|u\|_{L^{2s}(\Omega,\R^M)}^s \leq\|u\|_{L^2(\Omega,\R^M)}^{s(1-\tilde\theta_)}
\|u\|_{L^{2^{\ast\ast}}(\Omega,\R^M)}^{s\tilde\theta_1}.
\]
The continuity of the embedding $H^2(\Omega,\R^M)\hookrightarrow L^{2^{\ast\ast}}(\Omega,\R^M)$ and \eqref{eq:3.3} yield
\be\label{szac e}
I_1\leq c\|u\|_{H^2(\Omega,\R^M)}^{\tilde\gamma_1},\ee
where $\tilde\gamma_1:=s\tilde\theta_1 < 1$ since \eqref{eq:3.7} and $s<\frac{N+4}{N}$.\\
\indent We now estimate $I_2$: the interpolation inequality with exponents $2q\in [2,2^\ast)$, $\theta_2\in [0,1)$ such that
\begin{equation}\label{eq:3.5}
\frac{1}{2q}= \frac{1-\theta_2}{2}+ \frac{\theta_2}{2^\ast}\iff \frac{q}{2}\theta_2 = \frac{N}{4}(q-1)
\end{equation}
shows that
\[
I_2= \||\part u|\|_{L^{2q}(\Omega)}^q\leq \||\part u| \|_{L^2(\Omega)}^{q(1-\theta_2)}\||\part u|\|^{q\theta_2}_{L^{2^\ast}(\Omega)}.
\]
Both right hand side factors are estimated separately.
We have
\[
\||\part u| \|_{L^2(\Omega)}^{q(1-\theta_2)}= \|\part u \|_{L^2(\Omega,\R^M)}^{q(1-\theta_2)}\leq \| u \|_{H^1(\Omega,\R^M)}^{q(1-\theta_2)}
\]
We use the Ehrling--Browder inequality
\[
\|u\|_{H^1(\Omega,\R^M)}\leq c\|u\|_{H^2(\Omega,\R^M)}^{1/2} \|u\|_{L^2(\Omega,\R^M)}^{1/2}
\]
and \eqref{eq:3.3} to deduce
\[
\||\part u|\|_{L^2(\Omega)}^{q(1-\theta_2)}\leq c \|u\|_{H^2(\Omega,\R^M)}^{\frac{q}{2}(1-\theta_2)}M_0^{\frac{q}{2}(1-\theta_2)}.
\]
By \cite[Cor. 4.6]{ArendtKreuter}, $|\part u|\in H^1(\Omega)$ and $\| |\part u|\|_{H^1(\Omega)}\leq \|\part u\|_{H^1(\Omega,\R^M)}\leq c\|u\|_{H^2(\Omega,\R^M)}$.
Therefore and by the continuity of the embedding $H^1(\Omega)\subset L^{2^\ast}(\Omega)$
\[
\||\part u|\|_{L^{2^\ast}(\Omega)}^{q\theta_2}\leq c \| |\part u|\|_{H^1(\Omega)}^{q\theta_2}\leq c\|u\|_{H^2(\Omega,\R^M)}^{q\theta_2}.
\]
Summing up
\be\label{szac b}
I_2\leq c\|u\|_{H^2(\Omega,\R^M)}^{\gamma_2},\;\text{where}\; \gamma_2:=\frac{q}{2}(1+\theta_2).
\ee
The inequality $q<\frac{N+4}{N+2}$ together with \eqref{eq:3.5} yield $\gamma_2<1$.\\
\indent Using the regularity theory (see \cite[Theorem 4.14]{Giaquinta} and  \cite[Thm 8.12]{Gilbarg}, cf. Lemma \ref{lem:4.3})
$$
\|u\|_{H^2}\leq c(\|f(\cdot,u,\part u)\|_{L^2(\Omega,\R^M)} + \|u\|_{L^2(\Omega,\R^M)}),
$$
where the constant depends on $\Omega$, the ellipticity constant and $C^{0,1}$-norms of the coefficients of $\mathcal P_0$.
Therefore, in view of \eqref{eq:3.3}, \eqref{szac a}, \eqref{szac d}, \eqref{szac e} and \eqref{szac b}, we see that there is $c>0$ and $\gamma \in (0,1)$  such that every strong solution $u$ of \eqref{eq:3.1} and \eqref{eq:3.2} satisfies $\|u\|_{H^2(\Omega,\R^M)}\leq c\le 1+\|u\|_{H^2(\Omega,\R^M)}^\gamma\pr$
and the assertion follows. \end{proof}
Again by \eqref{szac a} along with  \eqref{szac d}, \eqref{szac e}, \eqref{szac b}, and Lemma \ref{lem:3.1} we have
\begin{corollary}\label{sols 0}
There is $C_2>0$ such that $\|f(\cdot,u,\part u)\|_{L^2(\Omega,\R^M)}\leq C_2$ \,for every strong solution $u$ of \eqref{eq:3.1} satisfying \eqref{eq:3.2}.\hfill $\square$
\end{corollary}

\begin{theorem}\label{thm:3.1}
There is a strong solution to the system \eqref{eq:3.1} satisfying the constraints \eqref{eq:3.2}.
\end{theorem}
\begin{proof} The proof relies on an abstract approach using the topological tools provided in Section \ref{appendix}.

\noindent {\bf Step\, I}: We introduce an abstract setting of the problem. Let $V:=H^1_0(\Omega,\R^M)$, $X:=L^2(\Omega,\R^M)$; both $V$ and $X$ are Hilbert spaces. Let
$\K\subset X$ be  the collection of {\em all} $L^2$-selections of $K(\cdot)$, i.e.,
\be\label{zbior}\K:=\{u\in X\mid u(x)\in K(x)\; \text{for a.a.}\; x\in\Omega\}.\ee
It has been already shown that $\K\neq\emptyset$ (comp. also \cite[Theorem 8.1.3]{Aubin}); $\K$ is closed and convex.\\
\indent Clearly $V$ is a dense subset of the  space $X$ and $V$ is continuously embedded into $X$, i.e., $\|v\|_X\leq c\|v\|_V$,  $v\in V$.  The bilinear form $\mathcal B:V\times V\to\R$ given by \eqref{form} is continuous and weakly coercive (see Remark \ref{cons of additional} (ii)).  Therefore, in view of the general construction in Subsection \ref{secHil} the {\em sectorial operator} $\A:D(\A)\to X$ such that $\la \A u,v\ra_X={\mathcal B}[u,v]$ for any $u\in D(\A)$ and $v\in V$ is well-defined.\\
\indent  Let us collect some properties of $\A$.

\indent (a) In view of the regularity results (see, e.g., \cite[Section 4.3]{Giaquinta}), the domain $D(\A) =H^2(\Omega,\R^M)\cap H_0^1(\Omega,\R^M)$  and $\A(u)={\mathcal P}[u]$ for $u\in D(\A)$, i.e., $\A$ is the $L^2$-realization of $\mathcal P$.

(b) $\A$ is closed, densely defined, the resolvent set $\rho(-\A)$ contains the set $\{\lambda\in\C\mid\Re\lambda>\omega\}$ and  $\|(\lambda I+A)^{-1}\|_{\mathcal L(X)}\leq(\lambda-\omega)^{-1}$ for $\lambda>\omega$, where $I$ is the identity on $X$. Given $h>0$, $h\omega<1$, the {\em resolvent} $J_h=J_h^{-\A}:=(I+h\A)^{-1}\in {\mathcal L}(X)$ is well-defined and
$J_h(X)\subset D(\A)$; see Lemma \ref{lem:0.2} for some additional properties of $J_h$.

(c)  $-\A$ generates a {\em analytic semigroup} $\{e^{-t\A}\}_{t\geq 0}$ of linear bounded operators on $X$ such that $\|e^{-t\A}\|_{\mathcal L(X)}\leq e^{\omega t}$ for $t\geq 0$. Observe that assumption \ref{additional ass} is equivalent to the semigroup invariance of $\K$, i.e.,
\be\label{si}e^{-t\A}(\K)\subset\K,\; t\geq 0,\ee
or the resolvent invariance of $\K$, i.e.,
\be\label{ri}J_h(\K)\subset \K,\; \text{where}\; h>0,\; h\omega<1.\ee
For the equivalence of the assumption \ref{additional ass} (2) and \eqref{si}, it is sufficient to  show that $u(t):=e^{-t\A}u_0$, $t\geq 0$, where $u_0\in\K$, is the unique weak solution to \eqref{parabolic}. The semigroup $\{e^{-t\A}\}$ is analytic, hence $u(t)\in D(\A)$, for $t>0$. Thus, $u\in C^1((0,+\infty),X)$ and $u'(t)=\A u(t)$ for  $t>0$. If $v\in V$, then by \cite[Cor. III.1.1]{Show}, for every $t>0$,
$$[u'(t),v]=\frac{d}{dt}\la u(t),v\ra_X=\la \A u(t),v\ra_X={\mathcal B}[u(t),v].$$
For the equivalence of \eqref{si} and \eqref{ri} see the proof Propostion \ref{invariance 1}; compare Section \ref{INV} for some other invariance issues.

(d) The compactness of the embedding $H_0^1(\Omega,\R^M)\hookrightarrow L^2(\Omega,\R^M)$ implies that  the semigroup $\{e^{-t\A}\}_{t\geq 0}$ is {\em compact} and so is the resolvent $J_h^{-\A}$, where $h>0$, $h\omega<1$ (\footnote{Let us  add that if $\Omega=\R^N$, then the above construction is valid, too, i.e., $\mathcal P$ determines the sectorial operator $\A$ (see \cite{Kunst}), but the  semigroup $\{e^{-t\A}\}$ is not compact in general.}).

\indent Now recall assumption \ref{nonlinear} and set
\be\label{def p}p := \max\left\{2q,\le \frac{1}{2s} + \frac{1}{N}\pr^{-1}\right\}.\ee
Then assumptions $s<\frac{N+4}{N}$ and $q<\frac{N+4}{N+2}$ imply that $2\leq p< \min\{2^\ast, N\}$. The $L^2$-realization $\A:D(\A)\to L^2(\Omega,\R^M)$ of $\mathcal P$ is sectorial, so, by inspection of the proof of \cite[Theorem 1.6.1]{Henry} there is $\alpha\in (0,1)$ such that
$$X^\alpha\subset W^{1,p}(\Omega,\R^M),$$
where $X^\alpha$ is the fractional space associated with $\A$  (see Section 4.1). The definition of $p$ and the Sobolev embeddings yield
\begin{equation}\label{eq:3.8}
X^\alpha \subset W^{1,p}(\Omega,\R^M)\subset L^{2s}(\Omega,\R^M) \cap W^{1,2q}(\Omega,\R^M).
\end{equation}
Recall \eqref{zbior} and define
\[
\K^\alpha:= \K\cap X^\alpha.
\]
By \eqref{ri}, $J_h(\K)\subset\K\cap D(\A)$ for sufficiently small $h>0$, so, in particular, $\K^\alpha\neq\emptyset$.  Let
$\F:\K^\alpha\to X$ be the superposition operator determined by $f$, i.e., for a.e. $x\in \Omega$ and  $u\in X^\alpha$
\[
\F(u)(x) = f(x,u(x), \part u(x)).
\]
The growth assumption $(f_2)$ and  \eqref{eq:3.8} show that, for $u \in \K^\alpha$,
\begin{align}
\label{szac1}\|\F(u)\|_{L^2(\Omega,\R^M)}&\leq c\left(\|\beta\|_{L^2(\Omega)} + \le \int_\Omega|u|^{2s}\d x \pr^{1/2} + \le \int_\Omega |\part u|^{2q} \d x \pr ^{1/2}\right)\\
&\nonumber\leq c(\|\beta\|_{L^2(\Omega)} + \|u\|_{L^{2s}(\Omega,\R^M)}^s +\|u\|_{W^{1,2q}(\Omega,\R^M)}^q)\leq
c(1 + \|u\|_{\alpha}^{\max\{s,q\}}).
\end{align}
This proves that $\F$ is well-defined.
The standard argument using the Lebesgue dominated convergence theorem shows that $\F$ is continuous. The {\em tangency} condition \eqref{tangency}  implies that
\be\label{tan} \F(u) \in T_\K(u)\;\; \text{for all}\;\; u\in \K^\alpha,\ee
where $T_\K(u)$ is the tangent cone to $\K$ at $u$ (see Section \ref{sto} and Example \ref{au}).
\begin{remark}\label{sols} {\em Clearly a function $u$ is a strong solution to \eqref{eq:3.1} satisfying \eqref{eq:3.2} if and only if $u\in \K\cap D(\A)$ and $\A u=\F(u)$.
}\end{remark}

\noindent {\bf Step\, II}: Taking into account issues collected above we are in a position to apply the constrained degree theory developed in Section \ref{constrained degree}. Namely:
\begin{enumerate}
\item[(a)] a sectorial operator $\A:D(\A)\to X$ with compact resolvent, $\alpha\in [0,1)$ and the fractional space $(X^\alpha,\|~\cdot\|_\alpha)$ associated with $\A$ are given;
\item[(b)] $\K\subset X$ is closed convex bounded and  $J_h(\K)\subset \K$ for all $h>0$ with $h\omega<1$;
\item[(c)] $\F:\K^\alpha\to X$ is continuous and tangent to $\K$, i.e., condition \eqref{tan} holds.
\end{enumerate}
\indent Let $C:=\{u\in\K^\alpha\cap D(A)\mid\A u=t\F(u),\; t\in [0,1]\}$. By  Remark \ref{sols}, \eqref{eq:3.3} and Corollary \ref{sols 0}, the sets $C$ and $\F(C)$ are bounded in $X$. $C$  is also closed in $\K^\alpha$ (see Remark \ref{uwagi} (2)). In view of Corollary \ref{wniosek kon} we have the conclusion of Theorem \ref{thm:3.1}. \end{proof}

\section{The problem on $\R^N$}

In this section we study the problem \eqref{main problem}, \eqref{cons}, where $\Omega=\R^N$. We will apply the approximation-truncation approach sketched in Introduction. For that reason we enhance assumption \ref{additional ass}.

\begin{assumption}\label{add ass}{\em \mbox{}

\begin{enumerate} \item The coefficients $A^{ij}\in \R^{M\times M}$, $1\leq i,j\leq N$, of $\mathcal P$ are constant and $\mathcal P$ is {\em strongly elliptic} in the sense of the {\em Legendre condition}, i.e., there is $\theta>0$ such that for any $\xi\in\R^{M\times N}$
    $$\sum_{i,j=1}^N\sum_{k,l=1}^MA^{ij}_{kl}\xi_{lj}\xi_{ki}\geq\theta|\xi|^2;$$
\item there is a sequence $(R_n)_{n=1}^\infty$ with  $R_n\nearrow\infty$ such that, for any $n\geq 1$, the graph $\Gr\left(K|_{B_n}\right)$ is invariant with respect to \eqref{parabolic} (with $\Omega$ replaced with $B_n$); here  $B_n:=\{x\in\R^N\mid|x|<R_n\}$ is the ball around 0 of radius $R_n$; in what follows, we assume that $R_n=n$, $n\geq 1$, for short.\end{enumerate}
}
\end{assumption}

\indent We start with some auxiliary lemmata.
\begin{lemma}\label{lem:4.1} {\em (i)} For every $1\leq p<N$, there is a constant $c_0=c_0(p)$ depending on $p$ only, such that for any $R\geq 1$ and $v\in W^{1,p}(B_R,\R^M)$
\be\label{sobolev constant}\|v\|_{L^{p\ast}(B_R,\R^M)}\leq c_0(p)\|v\|_{W^{1,p}(B_R,\R^M)}.\ee
\indent {\em (ii)} There is a constant $c_1>0$ such that for any $R\geq 1$ and $v\in H^2(B_R,\R^M)$,
\be\label{EB}|v|^2_{1,2,B_R}\leq c^2_1(\|v\|_{L^2(B_R,\R^M)}|v|_{2,2,B_R}+\|v\|_{L^2(B_R,\R^M)}^2).\ee
\end{lemma}
\begin{proof} These seem to be folklore results --- for the sake of completeness we give the proofs.\\ \indent (i) By the Sobolev inequality there is $c_0=c_0(p)$ such that
$$\|u\|_{L^{p^\ast}(B_1)}\leq c_0\|u\|_{W^{1,p}(B_1)},$$
for any $u\in W^{1,p}(B_1)$. Take $R\geq 1$, $v\in W^{1,p}(B_R)$ and let $u(x):=v(Rx)$ for $x\in B_1$.  Clearly, $u\in W^{1,p}(B_1)$ and, after the change of variables, we get
$$\|u\|_{L^{p^\ast}(B_1)}=R^{-N/p^{\ast}}\|v\|_{L^{p^\ast}(B_R)},\;\|u\|^p_{L^p(B_1)}=
R^{-N}\|v\|_{L^p(B_R)}^p\;\; \text{and}\;\;\|\part_ju\|_{L^p(B_1)}=R^{p-N}\|\part_jv\|^p_{L^p(B_R)}, $$
for any $j=1,\ldots,N$. Hence and taking into account that $Np/p^\ast=N-p$, we obtain
\begin{align}\label{przyd}\|v\|_{L^{p^\ast}(B_R)}&\leq c_0\left(R^{Np/p^\ast-N}\|v\|^p_{L^p(B_R)}+R^{Np/p^\ast+p-N}
\sum_{j=1}^N\|\part_jv\|^p_{L^p(B_R)}\right)^{1/p}\\
\nonumber &\leq c_0\|v\|_{W^{1,p(B_R)}}.\end{align}
If $v\in W^{1,p}(B_R,\R^M)$, then $w:=|v| \in W^{1,p}(B_R)$ and $\|w\|_{W^{1,p}(B_R)}\leq \|v\|_{W^{1,p}(B_R,\R^M)}$, in view of \cite[Cor 4.6]{ArendtKreuter}. Hence and by the above
\[
\|v\|_{L^{p^\ast}(B_R,\R^M)}= \|w\|_{L^{p^\ast}(B_R)} \leq c_0 \|w\|_{W^{1,p}(B_R)}\leq c_0 \|v\|_{W^{1,p}(B_R,\R^M)}.
\]

\indent (ii)
By the Ehrling-Browder inequalities (see \cite[Corollary 4.16, Theorem 4.17]{Adams}), for  $u\in H^2(B_1,\R^M)$
\[
\|u\|^2_{H^2(B_1,\R^M)}\leq c(|u|_{2,2,B_1}^2 + \|u\|_{L^2(B_1,\R^M)}^2)\;\;\text{and}\;\;
|u|^2_{1,2}\leq c\|u\|_{H^2(B_1,\R^M)}\|u\|_{L^2(B_1,\R^M)}.
\]
Combining the above inequalities, there is $c_1>0$ such that for any
$u\in H^2(B_1,\R^M)$
\begin{equation}\label{eq:4.2}
|u|^2_{1,2,B_1}\leq c^2_1 (|u|_{2,2,B_1}\|u\|_{L^2(B_1,\R^M)}+\|u\|^2_{L^2(B_1,\R^M)}).
\end{equation}
Fix $R\geq 1$ and $v\in H^2(B_R,\R^M)$. Again we define $u(x):=v(Rx)$ for $x\in B_1$.
Then $u\in H^2(B_1,R^M)$ and, after the change of variables,
$$|u|^2_{1,2,B_1}=R^{2-N}|v|^2_{1,2,B_R},\;|u|_{2,2,B_1}
=R^{2-N/2}|v|_{2,2,B_R},\;\|u\|_{L^2(B_1,\R^M)}=R^{-N/2}\|v\|_{L^2(B_R,\R^M)}.$$
Therefore and by \eqref{eq:4.2}
\[
|v|^2_{1,2,B_R} \leq c^2_1(|v|_{2,2,B_R}\|v\|_{L^2(B_R,\R^M)} + R^{-2}\|v\|^2_{L^2(B_R,\R^M)}.\qedhere
\]
\end{proof}
\begin{lemma}\label{lem:4.3}
There is $c_2 >0$ such that for all $R\geq 1$ and $g\in L^2(B_R,\R^M)$ if $v\in H^1_0(B_R,\R^M)$ is a weak solution of ${\mathcal P}[v]=g$, then $v\in H^2(B_R,\R^M)$ and
\be\label{EBd}
\|v\|_{H^2(B_R,\R^M)}\leq c_2(\|g\|_{L^2(B_R,\R^M)} + \|v \|_{L^2(B_R,\R^M)}).
\ee
\end{lemma}
\begin{proof} Take $R\geq 1$, $g\in L^2(B_R,\R^M)$, $v\in H^2(B_R,\R^M)$ and let ${\mathcal B}[v,\vp]=\int_{B_R}\la g,\vp\ra\d x$ for any test function $\vp\in C^\infty_0(B_R,\R^M)$. Define $f(x):=g(Rx)$, $u(x):=v(Rx)$ for $x\in B_1$ and take a test function $\psi\in C^\infty_0(B_1,\R^M)$. Let $\vp(x):=\psi(R^{-1}x)$, for $x\in B_R$. We see that $f\in L^2(B_1,\R^M)$, $u\in H^1_0(B_1,\R^M)$ and  $\vp\in C^\infty_0(B_R,\R^M)$. After the change of variables
\begin{gather*}\int_{B_R}\la g,\vp\ra\d x=R^N\int_{B_1}\la f,\psi\ra\d x,\\
{\mathcal B}[v,\vp]=R^{N-2}\int_{B_1}\sum_{i,j=1}^N\la A^{ij}\part_ju,\part_i\psi\ra\d x+
R^{N-1}\int_{B_1}\sum_{i=1}^N\la B^i_R\part_iu,\psi\ra\d x+R^N\int\la C_Ru,\psi\ra\d x,
\end{gather*}
where $B^i_R(x):=B^i(Rx)$, $C_R(x):=C(Rx)$ for $x\in B_1$ ($i=1,\ldots,N$). This shows that $u$ weakly solves the problem
$${\mathcal P}_0[u]=R^2f+T[u],\; \text{where}\; {\mathcal P}_0[u]:=-\sum_{i,j=1}^N\part_i(A^{ij}\part_ju)\;\; \text{and}\;\; T[u]:
=R\sum_{i=1}^NB^i_R\part_iu+R^2C_Ru.$$
It is clear that $T[u]\in L^2(B_1,\R^M)$ and
\be\label{przydat}\|T[u]\|_{L^2(B_1,\R^M)}\leq RN\max_{i=1,\ldots,N}\|B^i\|_{L^\infty}|u|_{1,2,B_1}+
R^2\|C\|_{L^\infty}\|u\|_{L^2(B_1,\R^M)}.\ee
\indent The regularity result (see, e.g.,
 \cite[Thm 4.14]{Giaquinta}) implies that there is a constant $c_3>0$ such that given $h\in L^2(B_1,\R^M)$ if $w\in H^1_0(B_1,\R^M)$ is a
weak solution to ${\mathcal P}_0[w]=h$, then
$$|w|_{2,2,B_1}\leq c_3\|h\|_{L^2(B_1,\R^M)}.$$
This implies that $u\in H^2(B_1,\R^M)$, in consequence $v\in H^2(B_R,\R^M)$,  and
$$|u|_{2,2,B_1}\leq c_3(R^2\|f\|_{L^2(B_1,\R^M)}+\|T[u]\|_{L^2(B_1,\R^M)}).$$
Hence and by \eqref{przydat}, there is a constant $c_4$ (depending on $\mathcal P$ only) such that
\be\label{EB1}|u|_{2,2,B_1}\leq c_4(R^2\|f\|_{L^2(B_1,\R^M)}+R^2\|u\|_{L^2(B_1,\R^M)}+R|u|_{1,2,B_1}).\ee
As before, after the change of variables we find
\begin{gather*}|u|_{2,2,B_1}=R^{2-N/2}|v|_{2,2,B_R},\;\|f\|_{L^2(B_1,\R^M)}=
R^{-N/2}\|g\|_{L^2(B_R,\R^M)},\\
\|u\|_{L^2(B_1,\R^M)}=R^{-N/2}\|v\|_{L^2(B_R,\R^M)}\; \text{and}\; |u|_{1,2,B_1}=R^{1-N/2}|v|_{1,2,B_R},
\end{gather*}
so \eqref{EB1} becomes
\be\label{EB2}|v|_{2,2,B_R}\leq c_4(\|g\|_{L^2(B_R,\R^M)}+|v|_{1,2,B_R}+\|v\|_{L^2(B_R,\R^M)}).\ee
We use Lemma \ref{lem:4.1} (ii) and the inequality $ab\leq \eps a^2+b^2/\eps$, for $a,b\geq 0$, to deduce
$$|v|_{1,2,B_R}\leq c_1(\eps|v|_{2,2,B_R}+(1+\eps^{-1})\|v\|_{L^2(B_R,\R^M)}).$$
Taking $\eps$ so that $c_1c_4\eps=1/2$ and returning to \eqref{EB2}, we conclude
\be\label{EB3}|v|_{2,2,B_R}\leq c(\|g\|_{L^2(B_R,\R^M)}+\|v\|_{L^2(B_R,\R^M)})\ee
for some $c$ independent of the choice of $R$. Now, combining \eqref{EB3} with \eqref{EB}, we get the assertion \eqref{EBd}.
\end{proof}
\begin{theorem}\label{main th}
Problem \eqref{main problem}, \eqref{cons}, where $\Omega=\R^N$, has a strong solution.
\end{theorem}
\begin{proof}
We consider the family of truncated problems
\be\label{on RN}
\begin{cases}
{\mathcal P}[u]=f(x,u,\part u)&\text{on }B_n,\; u|_{\part B_n}=0,\\
u(x)\in K(x) &\text{for a.e. } x\in B_n,
\end{cases}\ee
where $n\in\N$.  By Theorem \ref{thm:3.1}, for every $n\in\N$, there is a solution $u_n\in H^2(B_n,\R^M)\cap H^1_0(B_n,\R^M)$ such that $u_n(x)\in K(x)$ a.e. on $B_n$ since assumption \ref{add ass} entails assumption \ref{additional ass} with $\Omega$ replaced by $B_n$.

\noindent {\bf Step\, I}: We claim that the sequence $(\|u_n\|_{H^2(B_n,\R^M)})_{n=1}^\infty$ is bounded. Indeed, in view of $(K_3)$,
\be\label{L2 estimate}|u_n(x)|\leq m(x)\;\; \text{a.e. and thus}\;\; \|u_n\|_{L^2(B_n,\R^M)}\leq M_0:= \|m\|_{L^2(\R^N)}.\ee
Now we are going to establish the uniform $H^2$-estimate. By Lemma \ref{lem:4.3} and \eqref{L2 estimate}, there is $c_2$ such that for all $n\geq 1$
\be\label{H2 1}
\|u_n\|_{H^2(B_n,\R^M)}\leq c_2(\|f(\cdot,u_n,\part u_n)\|_{L^2(B_n,\R^M)}+M_0).
\ee
As in \eqref{szac a}, condition $(f_2)$ yields, for  $n\geq 1$,
\be\label{H2 2}
\|f(\cdot,u_n,\part u_n)\|_{L^2(B_n,\R^M)}\leq c(\|\beta\|_{L^2}+I^{(n)}_1+I^{(n)}_2),
\ee
where
$$
I^{(n)}_1=\left(\int_{B_n}|u_n|^{2s}\d x\right)^{1/2}=\|u_n\|_{L^{2s}(B_n,\R^M)}^s,\;\; I^{(n)}_2=\left(\int_{B_n}|\part u_n|^{2q}\d x\right)^{1/2}=\||\part u_n|\|_{L^{2q}(B_n,\R^M)}^q.
$$
We now proceed similarly as in the proof of Lemma \ref{lem:3.1} to estimate $I_1^{(n)}$ and $I_2^{(n)}$ but, to get constants independent of $n\geq 1$, we apply Lemmata \ref{lem:4.1} and \ref{lem:4.3}. Obvious modifications of arguments used to get \eqref{szac d}, \eqref{szac e} and \eqref{szac b} show that
$$I_1^{(n)}\leq c(1+\|u_n\|_{H^2(B_n,\R^M)}^{\gamma_1})\;\; \text{and}\;\;
I_2^{(n)}\leq c(1+\|u_n\|^{\gamma_2}_{H^2(B_n,\R^M)}),$$
for some constants $c>0$, $\gamma_1,\gamma_2\in (0,1)$ independent  of $n\geq 1$.
Combine this with \eqref{H2 2}, to find
\be\label{e 7a}
\|f(\cdot,u_n,\part u_n)\|_{L^2(B_n,\R^M)}\leq c(1+\|u_n\|^\gamma_{H^2(B_n,\R^M)}),
\ee
where constants $c>0$ and $\gamma\in(0,1)$ do not depend on $n\geq 1$. Hence, and by \eqref{H2 1}, there is $M_1>0$ such that
\be\label{e 8}\sup_{n\geq 1}\|u_n\|_{H^2(B_n,\R^M)}\leq M_1.\ee

\noindent {\bf Step\, II}:  From now on let us think of each $u_n$ as being extended to zero outside $B_n$. Since $u_n\in H^1_0(B_n,\R^M)$, we may assume that $u_n \in H^1(\R^N,\R^M)$ and $\|u_n\|_{H^1(B_n,\R^M)}=\|u_n\|_{H^1(\R^N,\R^M)}$  (\footnote{Note that in general $u_n\notin H^2(\R^N,\R^M)$ and this is the reason for some technical difficulties in what follows.}). We will show that the set $\{u_n\}_{n\geq 1}$ is relatively compact in $H^1(\R^N,\R^M)$.\\
\indent The idea is to decompose $\{u_n\} \subset \{ \chi_{B_R}u_n\} + \{(1-\chi_{B_R})u_n\}$, where $R>0$ is large enough and $\chi_{B_R}$ stands for the indicator function of $B_R$,
and show that the first set, being bounded in $H^2$, is compact in the $H^1(B_R,R^M)$ due to the Rellich--Kondrachov theorem, while the second one is contained in the arbitrarily small ball. In general, however, $\chi_{B_R}u_n\notin H^1(\R^N,\R^M)$, so we introduce a smooth function $\f_R:\R^N\to [0,1]$ having properties similar to those of $(1-\chi_{B_R})$.\\
\indent To this end consider a function $\f\in C^\infty(\R)$ such that  $0\leq \f\leq 1$, $\f(t)= 0$ for $t \leq 1$ and $\f(t) =1$ for $t \geq 4$.
For $R>0$, let $\f_R:\R^N\to \R$ be given by
$$
\f_R(x):=\f(R^{-2}|x|^2),\qquad x\in \R^N.
$$
Then $\f_R\in C^\infty (\R^N)$, $0\leq\f_R\leq 1$, $\f_R(x) = 0$ for $x\in B_R$ and $\f_R(x) = 1 $ for $x\in \R^N \setminus B_{2R}$.\\
\indent For any $R>0$ and $n\in\N$, $\f_Ru_n\in H^1_0(B_n,\R^M)$ so  test \eqref{on RN} with $\f_Ru_n$ and get
\begin{gather*}{\mathcal B}[u_n,\f_R u_n]=
\int_{\R^N}\left(\f_R\sum_{i,j=1}^N\la A^{ij}\part_ju_n,\part_iu_n\ra+\sum_{i,j=1}^N\part_i\f_R\la A^{ij}\part_ju_n,u_n\ra\right.+\\
+\left.\sum_{i=1}^N\f_R\la B^i\part_iu_n,u_n\ra+\f_R\la Cu_n,u_n\ra\right)\d x=\int_{\R^N}\f_R\la f(x,u_n,\part u_n),u_n\ra\d x,\end{gather*}
where we integrate over $\R^N$ since $\supp\f_Ru_n\subset \cl B_n$. The strong ellipticity $({\mathcal P}_2')$ implies that
\begin{gather*}
\theta\int_{\R^N}\f_R|\part u_n|^2\d x\leq
\int_{\R^N}\f_R\sum_{i,j=1}^N\la A^{ij}\part_ju_n,\part_iu_n\ra\d x=
\int_{\R^N}\f_R\la f(x,u_n,\part u_n),u_n\ra\d x+\\-
\int_{R^N}\sum_{i,j=1}^N\part_i\f_R\la A^{ij}\part_ju_n,u_n\ra\d x-\int_{R^N}\left(\sum_{i=1}^N\f_R\la B^i\part_iu_n,u_n\ra-\f_R\la Cu_n,u_n\ra\right)
\d x.
\end{gather*}
Consequently
\be\label{e 9}\theta\int_{|x|\geq 2R}|\part u_n|^2\d x \leq \theta\int_{\R^N}\f_R|\part u_n|^2\d x\leq I_0(n,R)+I_1(n,R)+I_2(n,R),\ee
where
\begin{gather*}I_0(n,R):=\int_{|x|\geq R}|f(x,u_n,\part u_n)||u_n|\d x,\; I_1(n,R):=\int_{R\leq |x|\leq 2R}|u_n|\sum_{i,j=1}^N|A^{ij}||\part_ju_n||\part_i\f_R|\d x,\\
I_2(n,R):=\int_{|x|\geq R}|u_n|\sum_{i=1}^N|B^i||\part_iu_n|+|C||u_n|^2\d x.
\end{gather*}
We estimate the right hand side summands. Firstly, we use \eqref{L2 estimate} and get
\begin{gather}\nonumber I_0(n,R)=\int_{|x|\geq R} |f(x,u_n,\part u_n)||u_n|\d  x=\int_{B_n\setminus B_r}
|f(x,u_n,\part u_n)||u_n|\d x\leq \\
\label{e 11}\leq\|f(\cdot,u_n,\part u_n)\|_{L^2(B_n,\R^M)}\left(\int_{|x|\geq R}m^2(x)\d x\right)^{1/2}\to 0 \text{ as } R\to \infty,
\end{gather}
since, in view of \eqref{e 7a} and \eqref{e 8}, the first factor above is bounded.\\
\indent Using the properties of $\f_R$, we obtain
\begin{gather*}I_1(n,R)\leq \max_{1\leq i,j\leq N}|A^{ij}|\int_{R\leq |x|\leq 2R}\left(\sum_{j=1}^N|u_n||\part_ju_n|\right)\left(\sum_{i=1}^N|\part_i\f_R|\right)\d x\leq\\
\leq N\max_{i,j}|A^{ij}|\int_{R\leq|x|\leq 2R}|u_n||\part u_n||\part\f_R|\d x\leq
N\sup_{t\in\R}|\f'(t)|\max_{i,j}|A^{ij}|\int_{R\leq|x|\leq 2R}|u_n||\part u_n|\frac{2|x|}{R^2}\d x\\
\leq\frac{4N}{R}\sup_{t\in\R}|\f'(t)|\max_{i,j}|A^{ij}|\|u_n\|_{L^2}\|u_n\|_{H^1(\R^N,\R^M)}
\leq \frac{4N}{R}\sup_{t\in\R}|\f'(t)|\max_{i,j}|A^{ij}|M_0M_1,
\end{gather*}
since, on account of \eqref{e 8}, $\|u_n\|_{H^1(\R^N,\R^M)}=\|u_n\|_{H^1(B_n,\R^M)}\leq \|u_n\|_{H^2(B_n,\R^M)}\leq M_1$. Hence
\be\label{e 10}I_1(n,R)\to 0\; \text{as}\; R\to\infty\; \text{for}\; n\in\N.\ee
Finally
\begin{gather}\nonumber I_2(n,R)\leq\max_{i}\|B^i\|_{L^\infty}\int_{|x|\geq R}|u_n|\sum_{i=1}|\part_iu_n|\d x+\|C\|_{L^\infty}\int_{|x|\geq R}|u_n|^2\d x\leq\\
\nonumber\leq \sqrt{N}\max_{i}\|B^i\|_{L^\infty}\int_{|x|\geq R}m(x)|\part u|\d x+\|C\|_{L^\infty}\int_{|x|\geq R}m^2(x)\d x\leq \\
\label{e 10a} \leq\left(\sqrt{N}\max_{i}\|B^i\|_{L^\infty}M_1+\|C\|_{L^\infty}\right)\int_{|x|\geq R}m^2(x)\d x\to 0\; \text{as}\; R\to\infty
\end{gather}
uniformly for $n\in\N$. \\
\indent By \eqref{e 9}, \eqref{e 11}, \eqref{e 10} and \eqref{e 10a}, we find that $\sup_{n\geq 1}|u_n|_{1,2,\R^N\setminus B_{2R}}\to 0$ as $R\to\infty$. Hence and again by \eqref{L2 estimate}
\be\label{e 12}\sup_{n\geq 1}\|u_n\|_{H^1(\R^N\setminus B_R,\R^M)}\leq\sup_{n\geq 1}\left(\|m\|_{L^2(\R^\setminus B_R)}^2+|u_n|^2_{1,2,\R^N\setminus B_R}\right)\underset{R\to\infty}{\longrightarrow}0.\ee
\indent Take an arbitrary $\eps>0$ and $R_0>0$ such that for $R\geq R_0$
$$\sup_{n\geq 1}\|u_n\|_{H^1(\R^N\setminus B_{R},\R^M)}<\eps.$$
Then, for any $n\geq 1$ and some constant $c>0$ independent of $n$ we have,
\be\label{e 13}\|\f_{R_0}u_n\|_{H^1(\R^N,\R^M)}\leq
c\|\f_{R_0}\|_{W^{1,\infty}(\R)}\|u_n\|_{H^1(\R^N\setminus B_{R_0},\R^M)}<\eps c(1+\|\part\f_{R_0}\|_{L^\infty(\R)}).\ee
For $n\geq 2R_0$, we have $(1-\f_{R_0})u_n\in H^2(B_{2R_0},\R^M)$ and
$$\|(1-\f_{R_0})u_n\|_{H^2(B_{2R_0},\R^M)}\leq c\|\f_{R_0}\|_{W^{2,\infty}(\R)}\|u_n\|_{H^2(B_n,\R^M)}\leq c\|\f_{R_0}\|_{W^{2,\infty}(\R)}M_1.$$
This shows that the set $\{(1-\f_{R_0})u_n\}_{n\geq 2R_0}$ is bounded in $H^2(B_{2R_0},\R^M)$.  In view of the Rellich--Kondrachov theorem it is relatively compact in $H^1(B_{2R_0},\R^M)$.  At the same time this set is contained in $H_0^1(B_{2R_0},\R^M)$; the latter space (if we think of its elements as being extended onto $\R^N$)  is closed in $H^1(\R^N,\R^M)$. Therefore $\{(1-\f_{R_0})u_n\}_{n\geq 1}$ is relatively compact in $H^1(\R^N,\R^M)$. Summing up, for every $\eps>0$ there is $R_0>0$ such that
$$\{u_n\}_{n\geq 1}\subset \{(1-\f_{R_0})u_n\}_{n\geq 1}+ \{\f_{R_0}u_n\}_{n\geq 1},$$
where the first set is relatively compact in $H^1(\R^N,\R^M)$  while the second one the ball $B_{H^1(\R^N,\R^M)}(0,\eps)$. This proves the claim.

\noindent {\bf Step\, III}: If $u_0$ is a cluster point of $\{u_n\}$, then $u_0$ is a strong solution to \eqref{main problem}, \eqref{cons}. Indeed, without loss of generality we may assume that $u_n\to u_0$ in $H^1(\R^N,\R^M)$ and in $L^{2^\ast}(\R^N,\R^M)$. Therefore  $u_n(x)\to u_0(x)$ and $\part u_n(x)\to\part u_0(x)$ for a.a. $x\in\R^N$; moreover there are $h_0\in L^{2^\ast}(\R^N)$ and $h_1\in L^2(\R^N)$ such that $|u_n|, |u_0|\leq h_0$ and $|\part u_n|, |\part u_0|\leq h_1$ a.e. on $\R^N$. \\
\indent It is clear that $u_0(x)\in K(x)$ for a.a. $x\in\R^N$. To see that $u_0$ is a weak solution take an arbitrary $\psi\in C^\infty_0(\R^N,\R^M)$. The $H^1$-continuity of $\mathcal B$ implies that
\be\label{e 14}{\mathcal B}[u_n,\psi]\to {\mathcal B}[u_0,\psi]\;\; \text{as}\; n\to\infty.\ee
The continuity of $f(x,\cdot,\cdot)$ for a.a. $x\in\R^N$ implies that
$$f(x,u_n(x),\part u_n(x))\to f(x,u_0(x),\part u_0(x))\;\; \text{a.e. as}\; n\to\infty$$
and, due to the growth conditions, for some $c>0$,
$$|f(x,u_n(x),\part u_n(x))-f(x,u_0(x),\part u_0(x))|\leq c\gamma(x)$$
for a.a. $x\in\R^N$, where $\gamma(x):=\beta(x)+h_0^s(x)+h_1^q(x)$, $x\in\R^N$.
H\"older's inequality with suitable exponents (see, e.g., Remark \ref{about sol}) shows that $\gamma(\cdot)|\psi|\in L^1(\R^N)$. Then, by the Lebesgue theorem
$$\int_{R^N}\la f(x,u_n,\part u_n),\psi\ra\d x\to\int_{R^N}\la f(x,u_0,\part u_0),\psi\ra\d x\; \text{as}\;\; n\to\infty.$$
Therefore $u_0$ is a weak solution to \eqref{main problem}, \eqref{cons}.\\
\indent  Now take an arbitrary bounded $\Omega\subset\R^N$ and $R>0$ such that $\Omega\subset B_R$. If $n\geq R$, then the restriction $w_n$ of $u_n$ to $\Omega$ belongs to $H^2(\Omega,\R^M)$. By \eqref{e 8}, $\sup_{n\geq R}\|w_n\|_{H^2(\Omega,\R^M)}<\infty$, thus (up to a subsequence) $(w_n)$ converges  weakly  to $w_0\in H^2(\Omega,\R^M)$ and $w_n\to w_0$ in $H^1(\Omega,\R^M)$. This implies that $w_0$ is the restriction of $u_0$ to $\Omega$. We have shown that the weak solution $u_0\in H^2(\Omega,\R^M)$, i.e., $u_0$ is a strong solution.\end{proof}

\section{Sectorial operators and constrained degree}\label{appendix}
Here we collect some relevant facts used throughout the paper. We discuss assumptions and provide some examples  as well as we present the construction of the coincidence degree.

\subsection{Tangent cones {\rm (see e.g. \cite[Chapter 4]{Aubin})}}\label{sto} Let $\K$ be a closed subset of a Banach space $(X,\|\cdot\|)$ and  $x\in \K$. The  {\em Clarke} (or {\em circatangent}) {\em cone} to the set $K$ at $u$ is defined by
$$T_\K(u):=\{v\in X\mid \lim_{h\to 0^+,\,y\to u,\,y\in\K}h^{-1}d(y+hv,\K)=0\},$$
where $d(x,\K)=\inf_{v\in\K}\|x-v\|$ for $x\in X$. Obviously, $T_\K(u)$ is a convex cone.
If $\K$ is convex, then
$$T_\K(u)=\cl {\bigcup_{h>0}h^{-1}(\K-u)}$$
is the cone {\em tangent} to $\K$ in the sense of convex analysis and $v\in T_\K(u)$ if and only if $p(v)\leq 0$ for any $p\in X^*$ such that $p(w-u)\leq 0$ for all $w\in\K$. Observe that if $u\in\mathrm{Int}\,\K$, then $T_\K(u)=X$.\\
\begin{example}\label{au} {\em Let $K(\cdot)$ satisfy assumption \ref{constraint}. If  $\K\subset X=L^2(\Omega,\R^M)$ is defined by \eqref{zbior}, then, given $u\in\K$,
$$v\in T_\K(u) \iff v(x)\in T_{K(x)}(u(x))\;\;\text{for a.a. } x\in\Omega,$$
in view of \cite[Corollary 8.5.2]{Aubin}.}\end{example}

\subsection{Sectorial operators {\rm (see, e.g., \cite[Chapter 1.3]{Cholewa})}}\label{sect}
Let $(X,\|\cdot\|)$ be a (real) Banach space. A closed densely defined linear operator $\A:X\supset D(\A) \to X$ is a \emph{sectorial} (of angle $<\pi/2$)  if there are $0<\phi<\pi/2$, $M\geq 1$ and $a\in \R$ such that the spectrum $\sigma(\A)$ of $\A$ is contained in the sector $S_{\phi,a}:=\{\lambda\in\C\mid \lambda=a+re^{i\theta}, r>0, |\theta|<\phi\}\cup\{a\}$
and for $\lambda\notin S_{\phi,a}$
$$\|(\lambda I -\A)^{-1}\|_{\Lin (X)}\leq M|\lambda - a|^{-1} (\footnote{$\Lin(X)$ denotes the space of bounded linear operators on $X$.}).$$
It is well-known that  $\A$ is a sectorial operator if and only if  $-\A$ generates the {\em holomorphic  semigroup} $\{e^{-t\A}\}_{t\geq 0}$ and one has $\|e^{-t\A}\|\leq M'e^{-at}$ for $t\geq 0$ and some $M'\geq 1$.\\
\indent  If a sectorial operator $\A$ is {\em positive}, i.e., $\Re\lambda>0$ for $\lambda\in \sigma(\A)$, then for any $\alpha> 0$ the improper integral
$$\A^{-\alpha} := \frac{1}{\Gamma (\alpha)}\int_0^\infty t^{\alpha-1}e^{-t\A}\d t$$
converges in the norm topology in $\Lin (X)$ and $\A^{-\alpha}$ is injective. Let $\A^{\alpha}:=(\A^{-\alpha})^{-1}: X^\alpha\to X$, where $X^\alpha:=\A^{-\alpha}(X)$ is
the \emph{fractional space} associated with $\A$. $X^\alpha$ is a Banach space endowed with the norm $\|x\|_\alpha:=\|A^\alpha x\|$, $x\in X^\alpha$.  We also put $X^0:=X$ and $\A^0:=I$,  the identity on $X$. For each $\alpha\geq 0$, $\A^\alpha$ is a densely defined closed linear operator; for all $0\leq\alpha\leq\beta$, the embedding $X^\beta\hookrightarrow X^\alpha$ is dense and continuous; it is compact provided $\A$ has {\em compact}  resolvents. Observe that $X^1=D(\A)$; the norm $\|\cdot\|_1$ is equivalent to the graph norm in $D(\A)$.\\
\indent If $\A$ is a sectorial operator, then there is $d\in\R$ such that $\A_d:=\A+dI$ is positive (e.g. $d> -a$). Hence, given $\alpha\geq 0$, we may consider the fractional space $X_d^\alpha$ associated with $\A_d$ endowed with the norm $\|x\|_\alpha=\|x\|_{\alpha,d}:=\|\A^\alpha_d x\|$, $x\in X^\alpha$.
\begin{remark}\label{rem:0.1}{\em Different choices of $d$ give the same fractional space and equivalent norms on it (see \cite[Theorem 1.4.6]{Henry}). This implies that for a  sectorial operator $\A$ the fractional space $X^\alpha$ is uniquely defined as a topological vector space: regardless the choice of the norm there is no ambiguity in topological terminology. When speaking of its norm $\|\cdot\|_\alpha$, however, one has to specify a suitable $d$. }\end{remark}

\indent Let $\omega:=-a$. If $\A:D(\A)\to X$ is a sectorial operator, then $\{\lambda\in\C\mid \Re\lambda>-\omega\}$ is contained in the resolvent set $\rho(-\A)$.  Given $h>0$ with $h\omega<1$,
\be\label{resolvent}J_h=J_h^{-\A}=(I+h\A)^{-1}:X\to X\ee
is well-defined and
\be\label{resolvent2}J_h(X)\subset D(\A).\ee
Let us collect several well-known properties of $J_h$.
\begin{lemma}\label{lem:0.2} If $h>0$ and $h\omega<1$, then:
\begin{enumerate}
\item[{\em (i)}] if $h'>0$ and $h'\omega<1$, then
$J_h=J_{h'}\le \frac{h'}{h}I+ \le 1-\frac{h'}{h}\pr J_h\pr$;
\item[{\em (ii)}]  $\|J_h\|_{\Lin(X^\alpha)}\leq M(1-h\omega)^{-1}$  for every $\alpha\geq0$;
\end{enumerate}
moreover for all $\alpha\in [0,1)$:
\begin{enumerate}
\item[{\em (j)}] $\|J_h\|_{\Lin(X,X^\alpha)} \leq ch^{-\alpha}(1-h\omega)^{\alpha-1
}$ for some $c>0$;
\item[{\em (jj)}] if $\A$ has compact resolvent, then $J_h \in \Lin(X,X^\alpha)$ is compact;
\item[{\em (jjj)}] $\|J_hx - x\|_\alpha \to 0$ as $h\to 0^+$, for every $x\in X^\alpha$;
\item[{\em (jv)}]  the map $X\times (0,\omega_0) \ni (x,t) \mapsto J_h x \in X^\alpha$ is continuous, where $\omega_0:=\infty$ if $\omega\leq 0$ and $\omega_0:=\omega^{-1}$ if $\omega>0$.\hfill $\square$
\end{enumerate}
\end{lemma}
\subsubsection{}\label{secHil}Sectorial operators in Hilbert spaces are generated in the following  way; see \cite[Section 7.3.2, Corollary 7.3.5]{Haase} or \cite[Theorem 2.18]{Yagi} for details. Let a Hilbert space $V$ be a dense subset of a Hilbert space $X$ and assume that the embedding $V\hookrightarrow X$ is continuous. Let a bilinear form ${\mathcal B}:V\times V\to\R$ be continuous, i.e., ${\mathcal B}[u,v]\leq c\|u\|_V\|v\|_V$, $u,v\in V$, and weakly coercive, i.e., there are $\omega\in\R$ and $\alpha>0$ such that \be\label{coercive}{\mathcal B}[v,v]+\omega\|v\|^2_X\geq\alpha\|v\|^2_V\; \text{for}\; v\in V.\ee
If $\mathcal A:V\to V^*$ is given by $[\mathcal A u](v):=\mathcal B(u,v)$, $u,v\in V$, then the part $\A:={\mathcal A}|_X$ of $\mathcal A$ in $X=X^\ast$, given by $\A u:=\mathcal A u$ for $u\in D(\A):=\{u\in V\mid\mathcal A u\in X^*\}$ is a sectorial operator (with $a=-\omega$). Clearly, for $u\in D(\A)$ and $v\in V$, $\la\A u,v\ra_X={\mathcal B}[u,v]$.

\subsection{Semigroup invariance}\label{INV} Let $\K\subset X$, where $X$ is a Banach space, be closed and convex and let $\A:D(\A)\to X$ be a sectorial operator.
\begin{proposition}\label{invariance 1} The following conditions are equivalent:
\begin{enumerate}
\item[{\rm (i)}] $\K$ is semigroup invariant, i.e., $e^{-t\A}(\K)\subset\K$ for all $t\geq 0$;
\item[{\rm (ii)}] $\K$ is resolvent invariant, i.e., $J_h(\K)\subset \K$ for $h>0$ with $h\omega<1$;
\item[{\rm (iii)}] $\K\cap D(\A)$ is dense in $\K$ and for every $u\in\K\cap D(\A)$, $-\A u\in T_\K(u)$.
    \end{enumerate}
\end{proposition}
\begin{proof} The equivalence (i) $\Leftrightarrow$ (ii) follows in view of the so-called Post-Widder formula \cite[Cor. III.5.5]{EngelNagel} and the integral representation of resolvents of $-\A$ in terms of the semigroup (see equality (1.13) in \cite[\S II]{EngelNagel}); see also \cite[Thm VI.1.8]{EngelNagel}.\\
\indent Assume (i), take $u\in\K$ and let $t_n\searrow 0$. Then $e^{-t_n\A}u\in \K\cap D(\A)$ since the semigroup is analytic. Evidently, $e^{-t_n\A}u\to u$ as $n\to\infty$, i.e., $\K\cap D(\A)$ is dense in $\K$. Now let $u\in\K\cap D(\A)$, then
$$-\A u=\lim_{n\to\infty}\frac{e^{-t_n\A}u-u}{t_n}\in\cl{\bigcup_{t>0}\frac{\K-u}{t}}=T_\K(u).$$
\indent Assume (iii) and let $u\in\K\cap D(\A)$. By the assumption $-\A u\in T_\K(u)$. Hence there are sequences $h_n\searrow 0$ and $v_n\to -\A u$ such that $u+h_nv_n\in\K$. Thus
$$h_n^{-1}d(e^{-h_n\A}u,\K)\leq h_n^{-1}\|e^{-h_n\A}u-(u+h_nv_n)\|=\|h_n^{-1}(e^{-h_n\A}u-u)-v_n\|\to 0.$$
This shows that $0\in T^{-\A}_\K(u)$, where $T_\K^{-\A}(u)$ is the so-called {\em Pavel cone} (see \cite[Def. 8.1.3]{Carja}) defined as
$$T_\K^{-\A}(u):=\{v\in X\mid\liminf_{t\to 0^+}t^{-1}d(e^{-t\A}u+hv,\K)=0\}.$$
Due to \cite[Theorem 8.5.5]{Carja}, $e^{-t\A}u\in\K$ for all $t\geq 0$. If $u\in\K$, then $u=\lim_{n\to\infty}u_n$, where $u_n\in\K\cap D(\A)$ and then, for $t\geq 0$,
\[e^{t\A}u=\lim_{n\to\infty}e^{-t\A}u_n\in\K.\qedhere\]
\end{proof}
A  result similar to the equivalence (i) $\Leftrightarrow$ (iii) has been established in \cite[Proposition 4.5]{Cann} by using different methods.

If $X$ is a Hilbert space and $\A$ is generated by a bilinear form $\mathcal B$ as in subsection \ref{secHil}, then we get the following results. Let $\pi_\K:X\to\K$ be the metric projection onto $\K$, i.e., for $u\in X$, $\|u-\pi_\K(u)\|_X=d(u,\K)=\inf_{w\in\K}\|u-w\|_X$. The projection $v=\pi_\K(u)$ is uniquely chracterized by
\be\label{rzut}\la u-v,w-v\ra_X\leq 0\; \text{for any}\; w\in \K.\ee
\begin{proposition}\label{invariance 2} The set $\K$ is resolvent invariant if and only if \be\label{warunek inv}\pi_\K(V)\subset V\;\;\text{and}\;\; {\mathcal B}[\pi_\K(u),u-\pi_\K(u)]\geq 0,\;\;\text{for every }u\in V.\ee
\end{proposition}
\begin{proof} Assume that $\K$ is resolvent invariant, take $u\in V$ and let $v:=\pi_\K(u)$.
For any $h>0$, $h\omega<1$, $J_hv\in D(\A)\cap \K$ and  $\A J_hv=h^{-1}(v-J_hv)$. By \eqref{rzut}, $\la v-J_hv,v-u\ra_X\leq 0$. Therefore
$$\mathcal B[J_hv,J_hv-u]=\la \A J_hv,J_hv-u\ra_X=h^{-1}\left(\la v-J_hv,v-u\ra-\|v-J_hv\|_X^2\right)\leq 0.$$ Thus $\mathcal B[J_hv,J_hv]\leq \mathcal B[J_hv,u]$ and by \eqref{coercive} and
\be\label{nierpom}\alpha\|J_hv\|^2_V\leq \mathcal B[J_hv,J_hv]+\omega\|J_hv\|_X^2\leq \mathcal B[J_hv,u]+\omega\|J_hv\|_X^2\leq
c\|J_hv\|_V\|u\|_V+\|J_hv\|^2_X.\ee

Take a sequence $h_n\to 0^+$. By Lemma \ref{lem:0.2} (jjj), $J_{h_n}v\to v$ in $X$. Hence, and in view of \eqref{nierpom}, the sequence $(J_{h_n}v)$ is bounded in $V$ and, up to a subsequence, weakly convergent in $V$ to some $w\in V$. The continuity $V\hookrightarrow X$ implies that $v=w$. This shows that $v=\pi_\K(u)\in V$. Next, in view of \eqref{rzut} and since $J_{h_n}v\in \K$ we have that for any $n\geq 1$
$$\mathcal B[J_{h_n}v,u-v]=\la \A J_{h_n}v,u-v\ra_X=h_n^{-1}\la v-J_{h_n}v,u-v\ra_X\geq 0.$$
The weak continuity of $\mathcal B[\cdot,u-v]$ implies $\mathcal B[v,u-v]=\lim_{n\to\infty}\mathcal B[J_{h_n}v,u-v]\geq 0$.\\
\indent Conversely, assume \eqref{warunek inv}, take $h>0$, $h\omega<1$ and $u\in \K$. Let $y=J_hu$. Then, on account of \eqref{warunek inv}, \eqref{resolvent2}, $\pi_\K(y)\in \K\cap V$ and $u=y+h\A y$. In view of \eqref{rzut}, \eqref{warunek inv} and \eqref{coercive}
\begin{align*}&0\geq\la u-\pi_\K(y),y-\pi_\K(y)\ra_X=\la y-\pi_\K(y)+h\A y,y-\pi_\K(y)\ra_X\\
&=\|y-\pi_\K(y)\|_X^2+h\la \A y,y-\pi_\K(y)\ra_X=\|y-\pi_\K(y)\|_X^2+h\mathcal B[y,y-\pi_\K(y)]\\
&=\|y-\pi_\K(y)\|^2_X+h\mathcal B[y-\pi_\K(y),y-\pi_\K(y)]+h\mathcal B[\pi_\K(y),y-\pi_\K(y)]\\
&\geq (1-h\omega)\|y-\pi_\K(y)\|^2_X+h\alpha\|y-\pi_\K(y)\|_V^2\geq 0.\end{align*}
This shows that $y=\pi_\K(y)\in \
K$. \end{proof}

\subsubsection{}\label{concrete invariance} Recall \eqref{retra}, \eqref{zbior} and  the setting in Step I of the proof of Theorem \ref{thm:3.1}. If $u\in X=L^2(\Omega,\R^M)$ and $w\in \K$, then $r(x,u(x))\in \K$ and $|u(x)-r(x,u(x))|\leq |u(x)-w(x)|$ for a.a. $x\in\Omega$; hence $\|u-r(\cdot,u(\cdot))\|_X\leq\|u-w\|_X$. This shows that $\pi_{\K}:X\to \K$ given by
\be\label{metric projection}\pi_{\K}(u)=r(\cdot,u(\cdot)),\;\; u\in X,\ee
is the metric projection of $X$ onto $\K$. By Proposition \ref{invariance 2} and \eqref{metric projection} we get a characterization of assumption \ref{additional ass} (2) announced in Remark \ref{cons of additional} (i).

\begin{corollary}\label{invariance 3} Condition (2) from assumption \ref{additional ass} is satisfied if and only if \begin{gather}\label{metinv}r(\cdot,u(\cdot))\in H^1_0(\Omega,\R^M)\;\;\text{and}\\
\label{warform}\mathcal B[r(\cdot,u(\cdot)),u-r(\cdot,u(\cdot))]\geq 0,\end{gather}
for every $u\in H^1_0(\Omega,\R^M)$.\hfill $\square$
\end{corollary}
Let us finally establish condtions sufficient for \eqref{metinv}.
\begin{proposition}\label{do inv} Suppose that, in addition to conditions stated in  assumption \ref{constraint}
\begin{align}\label{r4i}&r(\cdot,0)\in H^1_0(\Omega,\R^M)\;\; \text{and}\\
\label{r4ii}&r(\cdot,u)\in H^1(\Omega,\R^M)\;\; \text{for any}\;\; u\in\R^M.\end{align}
Then $r(\cdot,u(\cdot))\in H^1_0(\Omega,\R^M)$ for any $u\in H_0^1(\Omega,\R^M)$.
\end{proposition}
\noindent Condition \eqref{r4i} means that, in a  sense, $K(\cdot)$ has an extension onto $\cl\Omega$ and $0\in K(x)$ if $x\in\part\Omega$.
\begin{proof} Let $u\in H^1_0(\Omega,\R^M)$. In view of \eqref{r4ii} and \cite[Lemma 5]{MarcusMizel2} (see also \cite{MarcusMizel1}) $r(\cdot,u(\cdot))\in H^1(\Omega,\R^M)$.
If  $u\in C^\infty_0(\Omega,\R^M)$, i.e., $u$ vanishes outside a compact subset $C$ of $\Omega$, then letting $w(x):=r(x,u(x))-r(x,0)$, $x\in \Omega$, we see that $w\in  H^1(\Omega,\R^M)$ and $w(x)=0$ for $x\in\Omega\setminus C$. By results of \cite{Egert},  $w\in H_0^1(\Omega,\R^M)$ and, by \eqref{r4i}, $r(\cdot,u(\cdot))\in H_0^1(\Omega,\R^M)$.  In general
$u\in H^1_0(\Omega,\R^M)$ is the $H^1$-limit of $u_n\in C_0^\infty(\Omega,\R^M)$, so the result follows from the $H^1$-continuity of the Nemytski operator generated by $r$ (see \cite{MarcusMizel}).\end{proof}

\subsection{Examples of constraints}\label{co} We provide some examples of a constraint $K:\Omega\multi\R^M$, where $\Omega\subset\R^N$, having the properties studied above. In each of these examples we describe tangent cones showing the nature of the tangency hypothesis.

\begin{example}\label{przyklady}{\em (1) (Moving rectangle) Assume that $\sigma,\tau\in H^1(\Omega,\R^M)$, $\sigma\leq \tau$ and let
$$K(x):=[\sigma(x),\tau(x)]=\{w\in\R^M\mid \sigma_k(x)\leq w_k\leq\tau_k(x), k=1,\ldots,M\},\; x\in\Omega.$$
Such constraints has been studied, e.g., in \cite{Kuiper1, Kuiper2} in case $\Omega$ is bounded.
It is immediate to see that conditions from assumption \ref{constraint} are satisfied. For each $x\in\Omega$, the projection $r(x,\cdot)$ of $\R^M$ onto $K(x)$ is given by $r(x,\cdot)=(r_1(x,\cdot),\ldots,r_M(x,\cdot))$, where for $k=1,\ldots,M$ and $u=(u_1,\ldots,u_m)\in\R^M$ and
\begin{gather}\label{retrakcja1}r_k(x,u)=(\tau_k-\sigma_k-(u_k-\tau_k)^-)^+
+\sigma_k=\begin{cases}\sigma_k(x)&\text{if}\;\;
u_k<\sigma_k(x)\\
u_k&\text{if}\;\; \sigma_k(x)\leq u_k\leq \tau_k(x)\\
\tau_k(x)&\text{if}\;\; \tau_k(x)<u_k.\end{cases},\end{gather}
In view of \cite[Section 7.4]{Gilbarg}, $r_k(\cdot,u)\in H^1(\Omega)$, so \eqref{r4ii} is satisfied. If $\sigma_k|_{\part\Omega}\leq 0$ and $\tau_k|_{\part\Omega}\geq 0$ in the sense of trace ($k=1,\ldots,M$), then \eqref{r4i} holds, too.
Fix  $x\in \Omega$ and take $w\in K(x)$.
Then $T_{K(x)}(w)=\R^M$ if $\sigma_k(x)<w_k<\tau_k(x)$ ($k=1,\ldots,M$) and
$$v=(v_1,\ldots,v_M)\in T_{K(x)}(w)\;\iff\;\begin{cases}v_k\geq 0&\text{if}\; w_k=\sigma_k(x),\\
v_k\leq 0&\text{if}\; w_k=\tau_k(x).\end{cases}
$$
(2) (Tube) Let $K
\subset \R^M$ be closed convex and bounded, $b\in H^1(\R^N,\R^M)$,
$\alpha\in H^1(\R^N)$ with $\essinf\alpha>0$ and
$$K(x)=b(x)+\alpha(x)K,\; x\in\Omega.$$
Similar constraints were studied in \cite{schroder}. If $s:\R^M\to K$ is the metric projection onto $K$, then
\be\label{tube}r(x,u)=b(x)+\alpha(x)s(\alpha(x)^{-1}(u-b(x))),\; x\in\Omega, u\in\R^M.\ee
Clearly $K(\cdot)$ satisfies assumption \ref{constraint} and \eqref{r4i}; condition \eqref{r4i} holds if $-\alpha^{-1}(x)b(x)\in K$ for $x\in\part \Omega$ in the sense of trace. It is easy to see that $T_{K(x)}(w)=T_K(u)$, for a.a. $x\in\Omega$ and $w=b(x)+\alpha(x)u\in K(x)$, where $u\in K$.

\noindent (3) (Ellipsoidal funnel) Let $K$ be the closed unit ball in $\R^M$ and all entries of a matrix-valued map $E:\Omega\to \R^{M\times M}$ belong to $H^1(\Omega)$ and let $\essinf_{x\in\Omega}\det E(x)>0$. One may show that, for $u\in\R^M$, the map
$$\Omega\ni x\mapsto y(x,u)=\argmin_{y\in K}\left[\frac{1}{2}\la \prescript\intercal{}\!E(x)E(x)y,y\ra-\la\prescript\intercal{}\!E(x)u, y\ra\right],$$
is in $H^1(\Omega,\R^M)$. The funnel
$$K(x):=E(x)B,\;\;x\in\Omega,$$
where $B$ is the closed unit ball, consists of {\em ellipsoids} $K(x)$ centered at the origin. It satisfies assumption \ref{constraint} and \eqref{r4i}, \eqref{r4ii} in view of the explicit formula of the projection $r(x,u)=E(x)y(x,u)$, $x\in\Omega$, $u\in\R^M$.  Moreover,  $v\in T_{K(x)}(u)$ if and only if $\la E(x)^{-1}v,E(x)^{-1}u\ra\leq 0$, for a.a. $x\in\Omega$ and $u\in K(x)$.

\noindent (4) (Moving polyhderon) Suppose that  a set $P\subset\{p\in\R^M\mid |p|=1\}$  is at most countable and consider
\be\label{repr}K(x):=\bigcup_{p\in P}K_p(x), \;\; K_p(x):=\{u\in\R^M\mid \la p,u\ra\leq \xi_p(x)\},\;x\in\Omega,\ee
where $\xi_p\in H^1(\Omega)$ and $\xi_p|_{\part\Omega}\geq 0$ in the sense of trace (\footnote{Observe that this representation is fairly general, since for a proper closed convex subset $K$ of a separable Banach space $X$ there is a countable family $P$ of the dual  $X^*$ such that, for any $p\in P$,  $a_p:=\sup_{x\in K}\la p,x\ra<\infty$ and
$K=\{u\in X\mid \la p,u\ra\leq a_p\; \text{for all}\; p\in P\}$. The set consists of all {\em supporting functionals} of $K$.}).
Properties (1), (2) from assumption \ref{constraint}, \eqref{r4i} and \eqref{r4ii} are  satisfied. For $x\in\Omega$ and $u\in K(x)$, let $P(u)=\{p\in P\mid \la p,u\ra=\xi_p(x)\}$. Then $T_{K(x)}(u)=\{v\in \R^M\mid \la p,v\ra\leq 0,\;\forall\,p\in P(u)\}$, for a.a. $x\in\Omega$.}
\end{example}
\indent We have the following result akin to criteria from \cite[Theorem 4.1]{Chueh} and \cite[Theorem 14.7]{Smoller}.
\begin{proposition}\label{invev} Recall the operator $\mathcal P$ defined by \eqref{diffop} and satisfying assumptions \ref{differential operator} and \ref{additional ass} (1). Let $K(\cdot)$ be defined as in Example \ref{przyklady} (4) above. Assumption \ref{additional ass} (2) is fulfilled  if:\\
\indent {\em (i)} for all $1\leq i,j\leq N$, any $p\in P$ is an eigenvector of transposed matrices $\prescript\intercal{}\!A^{ij}$, $\prescript\intercal{}\!B^i$ and $\prescript\intercal{}C$, i.e.,
\be\label{eigen}\prescript\intercal\!\!A^{ij}(x)p=a^{ij}(x)p,\;
\prescript\intercal{}\!B^i(x)p=b^i(x)p,\;
\prescript\intercal{}C(x)p=c(x)p\ee
for a.a. $x\in\Omega$ and some functions $a^{ij},b^i,c:\Omega\to\R$, and\\
\indent (ii)  for any $p\in P$,  $\mathcal B[\xi_p(\cdot)p,\eta(\cdot)p]\geq 0$ for any $\eta\in H_0^1(\Omega)$, $\eta\geq 0$.
\end{proposition}
\begin{proof} It is clear that $\K=\bigcap_{p\in P}\K_p$, where $\K_p=\{u\in L^2(\Omega,\R^M)\mid u(x)\in K_p(x)\;\text{a.e.}\}$. Hence, to show the invariance of $\K$ it is sufficient to show the invariance of $\K_p$, $p\in P$. By Proposition \ref{invariance 2} (or Corollary \ref{invariance 3}) together with Proposition \ref{do inv}, it is enough to prove that for any $u\in H_0^1(\Omega,\R^M)$, $\mathcal B[\pi(u),u-\pi(u)]\geq 0$, where $\pi:=\pi_{\K_p}$ is the projection onto $\K_p$. The metric projection onto $K_p(x)$ is given by
\be\label{halfspace}r(x,u)=u-(\la u,p\ra-\xi_p
(x))^+p,\; u\in\R^M,\;\; x\in\Omega, u\in\R^M.\ee
Thus $\pi$ is given by the formula
$$\pi(u)=u-(\la p,u(\cdot)\ra-\xi_p)^+p,\; u\in L^2(\Omega,\R^M).$$
To simplify the notation let
$v:=(\la p,u(\cdot)\ra-\xi_p)^+$. Clearly, $v\in H_0^1(\Omega)$, $v\geq 0$ and $v=0$ off the set
 $\Omega_0:=\{x\in\Omega\mid \la p,u(\cdot)\ra>\xi_p\}$.  By \cite[Section 7.4]{Gilbarg}, for any $i=1,\ldots,N$, $\part_iv=\la p,\part_iu(\cdot)\ra-\part_i\xi_p$ on $\Omega_0$ and 0 elsewhere. Our assumptions yield the following equalities on $\Omega_0$
\begin{gather*}
\sum_{i,j=1}^N\la A^{ij}(\part_ju-\part_jv(\cdot)p),\part_iv(\cdot)p\ra
+\sum_{i=1}^N\la B^i(\part_iu-\part_iv(\cdot)p),v(\cdot)p\ra +\la C(u-v(\cdot)p), v(\cdot)p\ra\\
=\sum_{i,j=1}^N\la \part_ju-\part v_j(\cdot)p,\part_iv(\cdot)\!\prescript\intercal{}\!A^{ij}p\ra+
\sum_{i=1}^N\la\part_iu-\part_i(\cdot)p,v(\cdot)\!\prescript\intercal{}\!B^ip\ra+\la u,v(\cdot)\!\prescript\intercal{}Cp\ra\\= \sum_{i,j=1}^Na^{ij}\part_j\xi_p\part_iv+\sum_{i=1}^Nb^i\part_i\xi_pv+c\xi_pv\\=
\sum_{i,j=1}^N\la A^{ij}\part_j\xi_p(\cdot)p,\part_iv(\cdot)p\ra+\sum_{i=1}^N\la B^i\part_j\xi_p(\cdot)p,v(\cdot)p\ra+\la C\xi_p(\cdot)p,v(\cdot)p\ra.
\end{gather*}
This, in view of (ii), implies that
\[
\mathcal B[\pi(u),u-\pi(u)]=\mathcal B[u-v(\cdot)p,v(\cdot)p]=\mathcal B[\xi_p(\cdot)p,v(\cdot)p]\geq 0.\qedhere
\]
\end{proof}
Proposition \ref{invev} together with Proposition \ref{invariance 2} give  sufficient conditions for the flow invariance of $\mathcal P$. It, however, suggests  that
a convex closed set having the large collection of supporting functionals (e.g. an ellipsoid) is, in general, flow invariant only when $A^{ij}=a_{ij}I$, $B^i=b_iI$ and $C=cI$ for some $a_{ij}, b_i, c\in\R$, for $i,j=1,\ldots,N$. This also explains the setting concerning operators and constraining sets in \cite{Amann} and other papers mentioned in Introduction. As a further example  we have the following immediate result corresponding to M\"uller's conditions.
 \begin{corollary}\label{invev 1}
Suppose $K(\cdot)$ is given by Example \ref{przyklady} (1). If the operator $\mathcal P$ is diagonal, i.e., for each $1\leq i,j\leq N$, matrices of coefficients $A^{ij}$, $B^i$ and $C$ are diagonal, $A^{ij}_{kl}=\delta_{kl}a^{ij}_k$, $B^i_{kl}=\delta_{kl}b^i_k$ and $C_{kl}=\delta_{kl}c_k$ for $1\leq k,l\leq M$ ($\delta_{kl}$ stands for the Kronecker delta), then assumption \ref{additional ass} (2) is satisfied if
$${\mathcal B}_k[\sigma_k,\eta]\leq 0\;\; \text{and}\;\; {\mathcal B}_k[\tau_k,\eta]\geq 0\;\;\text{for any}\;\; \eta\in H^1_0(\Omega),\;\eta\geq 0,$$
where
$${\mathcal B}_k[u,v]:=\int_\Omega\left(\sum_{i,j=1}^Na^{ij}_k\part_j\part_iv+
\sum_{i=1}^Nb^i_k\part_iuv+c_kuv\right)\d x,\;\; u,v\in H^1_0(\Omega).$$
\end{corollary}
\begin{proof} Observe that $K(\cdot)$ has the representation \eqref{repr}
$$K(x)=\bigcup_{p\in P}K_p(x),$$
where  $P=\{e_1,\ldots,e_M,-e_1,\ldots,-e_M\}$, $e_k=(\delta_{1k},\ldots,\delta_{kM})$ and $\xi_{e_k}:=\tau_k$, $\xi_{-e_k}=-\sigma_k$, $k=1,\ldots,M$. It is clear that, for $k=1,\ldots,m$, $e_k$ is an eigenvector of $\prescript\intercal{}\!A^{ij}$, $\prescript\intercal{}\!B^i$ and $\prescript\intercal{}C$, and, for every $\eta\in H^1_0(\Omega)$, $\eta\geq 0$,
${\mathcal B}[\xi_{e_k}(\cdot)e_k,\eta(\cdot)e_k]={\mathcal B}_k[\tau_k,\eta_k]\geq 0$ and
${\mathcal B}[\xi_{-e_k}(\cdot)(-e_k),\eta(\cdot)(-e_k)]=-{\mathcal B}_k[\sigma_k,\eta]\geq 0$.
\end{proof}

\subsection{Construction of the constrained topological degree}\label{constrained degree}
Let $X$ be a real Banach space.
We provide the construction of a topological invariant detecting constrained coincidences of $\A$ and $\F$, i.e., solutions to
\[
\A u = \F(u)
\]
in a set  $\K\subset X$ of constraints, where
\begin{enumerate}
\item[$(D_1)$] $\A:D(\A)\to X$ is a sectorial operator with compact resolvents,
\item[$(D_2)$] $\F:U\cap\K\to X$ is a continuous map, where $U$ is an open subset of the fractional space $X^\alpha$ associated with $\A$, where $\alpha\in [0,1)$.
\end{enumerate}
As mentioned in Introduction, in the studied situation the direct use of the Leray--Schauder theory is not possible. The approach we present has been started in \cite{CwiszewskiKryszewski2}. Then, relying on the ideas of \cite{Karsatos}, it has been developed in \cite{CwiszewskiKryszewski}, where $\F$ was defined on an open subset of $X$ (if $\alpha=0$). The need to consider a more general case, when $\F$ is defined on a fractional space stems from applications, e.g., those discussed in the beginning of this paper, where the right-hand side of a differential problem depends on the gradient of an unknown function.
This was already observed in \cite{KryszewskiSiemianowski}, but  $\K$ and $\F$ are assumed to be bounded there.
These assumptions are so strong that results of \cite{KryszewskiSiemianowski} are hardly applicable to differential systems \eqref{main problem}.
Therefore, it is necessary to modify the approach from \cite{CwiszewskiKryszewski} accordingly.\\
\indent Simple examples show that in order to get  meaningful results one should impose some structural assumptions. We assume that
\begin{enumerate}
\item[$(D_3)$] $\K\subset X$ is an $\Lin$-retract (\footnote{Even though our applications concern the convex $\K$ we decided to provide the construction in a more general situation having the future reference on mind.}).
\end{enumerate}
Recall that a closed set $\K$ of a Banach space $X$ is called an {\em $\Lin$-retract} (see \cite{BenK}) if there is $\eta>0$, a continuous map $r:B(\K,\eta)\to\K$, where $B(\K,\eta):=\{x\in X\mid d(x,\K)<\eta\}$, and $L\geq 1$ such that $\|r(x)-x\|\leq Ld(x,\K)$ for any $x\in B(\K,\eta)$. $\Lin$-retracts constitute a broad subclass of neighborhood retracts containing many classes of sets considered as constraint sets; in particular {\em any} closed convex set $\K\subset E$ is an $\Lin$-retract; see \cite{BenK} for more details and other examples of $\Lin$-retracts. \\
\indent Moreover we assume that
\begin{enumerate}
\item[$(D_4)$] $\K$ is resolvent invariant, i.e., $J_h(\K)\subset \K$ for all $0<h\leq h_0$, where $h_0\omega<1$;
\item[$(D_5)$] $\F$ is tangent to $\K$, i.e., $\F(u)\in T_\K(u)$ for $u\in U\cap \K$.
\end{enumerate}
Let
$$C=\Coin(\A,\F;U\cap\K):=\{u\in U\cap\K\mid u\in D(\A)\; \text{and}\; \A u=\F(u)\}.$$
The inclusion $C\subset D(A)$ implies that $C\subset X^\beta$ for every $\beta\in [0,1]$. We assume that
\begin{enumerate}
\item[$(D_6)$] $C$ is closed in $X^\alpha$, $C$ and $\F(C)$ are bounded in $X$.
\end{enumerate}
\begin{remark}\label{uwagi}{\em (1)  Assumption $(D_6)$ holds if and only if $C$  is {\em compact} in $X^\beta$, for every $\beta\in [0,1)$. Indeed, if $C$ and $\F(C)$ are bounded, then $C$ is relatively compact in $X^\alpha$ since
$C\subset \{J_h(u+h\F(u))\mid u\in C\}\subset X^\alpha$ and  $J_h:X\to X^\alpha$ is compact by Lemma \ref{lem:0.2}.   For $\beta\in [\alpha,1)$, $C$ is closed in $X^\beta$, since $X^\beta\subset X^\alpha$ is continuous, and thus compact in $X^\beta$. If $\beta\in [0,\alpha)$, then $C$ is compact in $X^\beta$ due to the continuity of the embedding $X^\alpha\subset X^\beta$.\\
\indent (2) Observe that $C$ is always closed in $U$ since $\F$ is continuous and $\A$ has the closed graph. Evidently if $\K^\alpha\subset U$, then $C$ is closed in $X^\alpha$.\\
\indent (3) Suppose that $\K$ is bounded. Then assumption $(D_4)$, along with $(D_1)$ and $(D_3)$, implies that, for any $\beta\in [0,1)$, the set $\K^\beta:=\K\cap X^\beta$ is of {\em finite homological type}, i.e., for each $q\geq 0$ the vector space $H_q(\K^\beta)$, where $H_*(\cdot)$ stands for the singular homology functor with the rational coefficients, is finite dimensional and $H_q(\K^\beta)=0$ for almost all $q\geq 0$. To see this let $O:=B(\K,\eta)$ and let $\phi(x):=J_h\circ r(x)$, $x\in O$, where $0<h\leq h_0$ is fixed. Then $\phi:O\to\K^\beta$ is a well-defined continuous compact map. Let $j:\K^\beta\hookrightarrow O$ be the embedding and $\bar\phi:=j\circ\phi:O\to O$. In view of the so-called normalization property of the Leray--Schauder fixed point index $\bar\phi$ is a Lefschetz map (see \cite[Theorem (7.1)]{Granas}) and its generalized Lefschetz number $\Lambda(\bar\phi)$ is defined.  The  commutativity of the following diagram
$$\xymatrix{O\ar[r]^-\phi\ar[d]^-{\bar\phi}&{\K^\beta}\ar@<-2pt>[d]^-{\phi|_{\K^\beta}}
\ar[dl]_-j\\
O\ar[r]_-{\phi}&{\K^\beta}}$$
implies that $\phi|_{\K^\beta}$ is a Lefschetz map (see \cite[Lemma (3.1)]{Granas}) and $\Lambda(\phi|_{\K^\beta})=\Lambda(\bar\phi)$. We show below that $\phi|_{\K^\beta}$ is homotopic to the identity $\id:\K^\beta\to\K^\beta$ and, thus, the endomorphisms $H_*(\phi|_{\K^\beta})=H_*(\mathrm{id})=\mathrm{id}_{H_*(\K^\beta)}$
are the Leray endomorphisms. Hence, the graded vector space $H_*(\K^\beta)$ is of finite type. In particular the {\em Euler characteristic} $\chi(\K^\beta):=\sum_{q\geq 0}(-1)^q\dim_\Q H_q(\K^\beta)$ is a well-defined integer number. Moreover
$\chi(\K^\beta)=\lambda(\id)=\Lambda(\id)=\Lambda(\bar\phi)$, where $\lambda(\id)$ is the usual Lefschetz number, does not depend on $\beta\in [0,1)$. \\
\indent To show that $\phi|_{\K^\beta}$ is homotopic to the identity consider a map
$\Phi:\K^\beta\times [0,1]\to\K^\beta$ given by
$$\Phi(x,t):=\begin{cases}J_{th}(x),&\text{for}\; x\in \K^\beta,\; t\in (0,1],\\
x,&\text{for}\; x\in\K^\beta,\; t=0.\end{cases}$$
In view of Lemma \ref{lem:0.2} (jv), $\Phi$ is continuous on $\K^\beta\times (0,1]$, so let us consider sequences $x_n\to x_0$ in $\K^\beta$ and $t_n\searrow 0$. By Lemma \ref{lem:0.2} (ii) and (jjj), we have
\be\label{d}\begin{split}\|\Phi(x_n,t_n)-\Phi(x_0,0)\|_\beta=\|J_{t_nh}x_n - x_0\|_\beta\leq \|J_{t_nh}x_n - J_{t_n h}x_0\|_\beta + \|J_{t_nh}x_0 - x_0\|_\beta\\
\leq \frac{M}{1-t_n h\omega} \|x_n - x_0\|_\beta + \|J_{t_nh}x_0 - x_0\|_\beta \to 0,\quad n\to \infty.\end{split}\ee
Evidently $\Phi$ is a homotopy joining $\phi|_{\K^\beta}$ to $\id$.
}
\end{remark}

Let us now present the steps of the construction.

\indent {\em Step 1}: Let $\tilde\F:U\to X$ be a continuous extension of $\F$. Note that  $\F$ is defined on a closed subset of $U$, so $\tilde\F$ exists  in view of the Dugundji extension theorem.\\
\indent {\em Step 2}: Fix $\eta >0$, $L\geq 1$ and a retraction $r:B_X(K,\eta)\to K$ such that $\|r(x) - x\| \leq L d(x,\K)$ for $x\in B_X(K,\eta)$.\\
\indent {\em Step 3}: Since $C$ is compact in $X^\alpha$  and $\tilde\F$ is continuous, one can find an open bounded subset $W\subset X^\alpha$ such that
\begin{equation}\label{eq:1.2}
C\subset W\subset \cl W \subset U\cap (B_X(\K,\eta/2)\cap X^\alpha),
\end{equation}
and $\tilde\F(\cl W)$ is bounded in $X$, where $\cl W$ is the closure of $W$ in $X^\alpha$. Note that $B_X(\K,\eta/2)\cap X^\alpha$ contains $C$ and, since the embedding $X^\alpha\hookrightarrow X$ is continuous, $B_X(\K,\eta/2)\cap X^\alpha$ is open in $X^\alpha$.\\
\indent {\em Step 4}: Since the set $\tilde\F(\cl W)$ is bounded we may assume that $\|h\tilde\F(x)\|\leq\eta/2$ for $h\in (0,h_0]$ and $x\in\cl W$. By \eqref{eq:1.2}, for any $x\in W$, $d(x,\K)<\eta/2$ and, thus,
$x+h\tilde\F(x)\in B_X(\K,\eta)$. Therefore and in view of $(D_4)$ the map $\phi_h:\cl W\to X^\alpha$, where $h\in (0,h_0]$, given by the formula
\[
\phi_h (x) = J_h\circ r (x + h\tilde\F(x)),\qquad x\in \cl W,
\]
is well-defined.
The map $\phi_h$ is compact due to the fact that $\A$ has compact resolvents and $r$ maps bounded sets into bounded ones.
Moreover, if $h>0$ is small enough, then $\Fix(\phi_h):=\{ x\in \cl W \mid \phi_h(x) = x\} \subset W$.
Indeed, if $x\in\Fix(\phi_h)$, then $x\in\cl W\cap \K$ since $J_h(\K)\subset \K$. Hence $\tilde\F(x)=\F(x)$ and $x=J_h(r(x+h\F(x)))$. Suppose to the contrary that there are  a sequences $h_n\searrow 0$ and $(x_n)\subset\part W\cap\Fix(\phi_n)$. Hence $x_n+h_n\A(x_n)=r(x_n+h_n\F(x_n))$. Arguing as in the proof of \cite[Lemma 3.3]{CwiszewskiKryszewski} (with obvious modifications) we gather that, after passing to a subsequence if necessary, $x_n\to x_0\in C\cap\part W$. This contradiction proves our assertion. \\
\indent {\em Step 5}: Without loss of generality we may assume that $\phi_h$ is well-defined compact and $\Fix(\phi_h)\subset W$ for any $h\in (0,h_0]$. This implies that for any such $h$, the {\em Leray--Schauder fixed point index} $\Ind_{LS}(\phi_h,W)$ is well-defined (see \cite[Sections 7, 8]{Granas}). Therefore, we are in a position to define the {\em constrained  topological degree} $\deg_\K(\A,\F;U)$ {\em of coincidence} between $\A$ and $\F$ on $U$  as follows:
\be\label{defdeg}
\deg_\K(\A,\F;U) : =\lim_{h\to 0^+} \Ind_{LS}(\phi_h, W).
\ee
We claim that the definition \eqref{defdeg} is correct, i.e.,
\begin{itemize}
\item it stabilizes, i.e., $\Ind_{LS}(\phi_{h_1},W)=\Ind(\phi_{h_2},W)$ for sufficiently small $h_1$, $h_2>0$,
\item it  is independent of the choice of an $\Lin$-retraction $r$,
\item it is independent of the choice of a bounded open neighbourhood $W$ of $C$,
\item it is independent of the extension $\tilde\F$.
\end{itemize}
These issues can be shown in a similar manner as in the proof of  \cite[Lemma 3.5 and Lemma 3.6]{CwiszewskiKryszewski}.\\
\indent The definition \eqref{defdeg} also does not depend on the choice of a constant $d$, where $d$ is such that $\A_d=\A+d I$ is positive and determines the norm $\|\cdot\|_{\alpha^,d}$ on the fractional space $X^\alpha$. Indeed, if we take another $\hat d>-\omega$, then the identity provides a (topological) homeomorphism between $(X^\alpha,\|\cdot\|_{\alpha,d})$ and $(X^\alpha,\|\cdot\|_{\alpha,\hat d})$ in view of Remark \ref{rem:0.1}. The claim  follows now immediately from the so-called {\em commutativity property} of the Leray--Schauder index (see \cite[Theorem (7.1)]{Granas}). One gets in fact a stronger result.
\begin{remark} {\em The so-called contraction property of the Leray--Schauder index (see \cite[Theorem (6.2), $\S$ 12]{GranasDugundji}) implies that for all $h\in (0,h_0]$,
$\Ind_{LS}(\phi_h,W)=\Ind_{LS}(\phi_h|_{W\cap X^\beta},W\cap X^\beta)$, where $\beta\in [\alpha,1)$. This implies that $\deg_\K(\A,\F,U)$ does not depend on the particular choice $\alpha\in [0,1)$, namely, if $\beta\in [\alpha,1)$, then $\deg_\K(\A,\F;U)=\deg_\K(\F|_{U\cap X^\beta\cap\K};U\cap\ X^\beta)$.
}
\end{remark}

By an {\em admissible homotopy} we understand a continuous map $\H: (U\cap \K)\times [0,1] \to X$, where $U\subset X^\alpha$ is open, such that $H(u,t) \in T_\K(u)$ for all $u\in U\cap \K$, $t\in [0,1]$, the sets $C=\bigcup_{t\in [0,1]}\Coin(\A,\H(\cdot,t);U\cap \K)$ and $\H(C\times [0,1])$ are bounded in $X$ and $C$ is closed in $X^\alpha$.
\begin{remark}\label{hom} {\em Arguing as above we easily see $C$ is compact in $X^\alpha$ and one  can find an open bounded $W\subset X^\alpha$ such that
$C\subset W \subset \cl W\subset U\cap (B_X(\K,\eta/2)\cap X^\alpha)$ and $H(\cl W \times  [0,1])$ is bounded in $X$. Then, for sufficiently small $h>0$ the map $\phi_h: \cl W\times [0,1] \to X^\alpha$ given by $\phi_h (x,t):= J_h\circ r(x + h\tilde\H(x,t))$ for $x\in \cl W$, $t\in [0,1]$, where $\tilde\H:U\to X$ is a continuous extension of $\H$,  is well-defined compact and
\[
\{x\in \cl W\mid \exists t \in [0,1],\;\phi_h(x,t) =x\}\subset W
\]
provided $h>0$ is sufficiently small, see \cite{CwiszewskiKryszewski}.}
\end{remark}

\begin{theorem}\label{thm:1.1}
The degree defined by \eqref{defdeg} has the following properties:

\noindent \emph{(Existence)} If\, $\deg_\K(\A,\F;U)\neq 0$, then $\Coin(\A,\F;U\cap\K)\neq\emptyset$.\\
\emph{(Additivity)} If\, $U_1$, $U_2\subset U$ are open disjoint and $\Coin(\A,\F;U)\subset (U_1\cup U_2)\setminus (U_1\cap U_2)$, then
\[
\deg_\K(\A,\F; U)=\deg_\K(\A,\F;U_1) +\deg_\K(\A,\F,U_2).
\]
\emph{(Homotopy invariance)} If\, $H:U\times [0,1] \to X$ is an admissible homotopy, then
\[
\deg_\K(\A,\H(\cdot,0),U)= \deg_\K(\A,\H(\cdot,1),U).
\]
\emph{(Normalisation)} If\, $\K$ is bounded, $\F:\K^\alpha\to X$ and $\F(\K^\alpha)$ is bounded in $X$, then for any open $U\subset X^\alpha$ such that $\K^\alpha\subset U$,
\[
\deg_\K(\A,\F,U)=\chi(\K).
\]
\end{theorem}

\begin{proof}
(Existence) By the definition \eqref{defdeg}, given a sequence $h_n\searrow 0$, we have $\Ind_{LS}(\phi_{h_n},W)\neq 0$. The existence property of the Leray--Schauder index implies the existence of a sequence $(x_n)$ in $W$ such that $\phi_{h_n}(x_n)=x_n$, i.e., $x_n\in D(\A)\cap W\cap\K$ and $x_n + h_n \A x_n = r( x_n + h_n\tilde\F(x_n))=r(x_n+h_n\F(x_n))$. Hence
\begin{equation}\label{eq:1.5}
\begin{aligned}
\|\A x_n-\F(x_n)\|& = \frac{1}{h_n} \|x_n +h_n\A x_n - x_n -h_n \F(x_n)\|\\&=\frac{1}{h_n}\|r(x_n+h_n\F(x_n))-(x_n+h_n\F(x_n))\|\leq
L\frac{d(x_n+h_n\F(x_n),\K)}{h_n},
\end{aligned}
\end{equation}
where in the last inequality we used the fact that $r$ is $\Lin$-retraction.
Since $x_n \in W\cap \K$, we get
\begin{equation}\label{eq:1.6}
\frac{d(x_n +h_n\F(x_n),\K)}{h_n}= \frac{d(x_n +h_n\F(x_n),\K)-d(x_n,\K)}{h_n}\leq \|\F(x_n)\|\leq R,
\end{equation}
for a constant $R$ such that $\sup_{x\in \cl W}\|F(x)\|= R$.
Equation \eqref{eq:1.5} combined with \eqref{eq:1.6} yields $\{\|Ax_n\|\}_{n\geq 1}$ is bounded. For any $n\geq 1$, $x_n=J_{h_0}(x_n +h_0\A x_n)$, and so, by Lemma \ref{lem:0.2} (jj),  the set $\{x_n\}_{n\geq}$ is relatively compact in $X^\alpha$. Passing to a subsequence if necessary, we have $x_n\to x_0$ in $X^\alpha$ and $x_0\in \cl W\cap\K\subset U\cap \K$.
By \eqref{eq:1.5}, we have
\[
\|\A x_n -\F(x_n)\|\leq L\frac{d(x_n+h_n\F(x_0),\K)}{h_n} + L\| \F(x_0)-\F(x_n)\|.
\]
Letting $n\to\infty$, using the continuity and the tangency $(D_5)$ of $\F$, we see that $\lim_{n\to\infty} \|\A x_n-\F(x_n)\|=0$.
Thus, $\A x_n\to \F(x_0)$ and, since $\A$ is closed, $x_0\in D(\A)$ with $\A x_0=\F(x_0)$.
Thus, $x_0 \in D(\A) \cap U\cap\K$ satisfies $\A x_0=\F(x_0)$.\\
\indent (Additivity) follows by the definition \ref{defdeg} from the additivity property of the Leray--Schauder index.\\
\indent  (Homotopy invariance)  is a consequence of the homotopy invariance of the Leray--Schauder index, the definition \ref{defdeg} and Remark \ref{hom}.\\
\indent (Normalization) Obviously, $\deg_\K(\A,\F;U)$ is well-defined and its independence of $U$ follows immediately from the additivity property. Take an open $W\subset X^\alpha$ such that $\K^\alpha\subset W \subset  \cl W\subset U\cap (B_X(\K^\alpha,\eta/2)\cap X^\alpha$) and a continuous  extension $\tilde\F:X^\alpha\to X$ of $\F$ being bounded on $\cl W$. Let $h\in (0,h_0]$ be such that $\deg_\K(\A,\F;U)=\Ind_{LS}(\phi_h,W)$. Recall that $\phi_h(W)\subset\K^\alpha$. Hence, we may assume that $\phi_h:W\to \K^\alpha$.  Denoting the embedding $\K^\alpha\hookrightarrow W$ by $j$ and $\bar\phi_h:=j\circ \phi_h:W\to W$, we see that $\bar\phi_h$ is a Lefschetz map and
\be\label{lef1}\Ind_{LS}(\phi_h,W)=\Ind_{LS}(\bar\phi_h,W)=\Lambda(\bar\phi_h)\ee is the generalized Lefschetz number of $\bar\phi_h$. The argument is now similar to that from Remark \ref{uwagi} (3). We have the commutative diagram
$$\xymatrix{W\ar[r]^-{\phi_h}\ar[d]^-{\bar\phi_h}&{\K^\alpha}
\ar@<-2pt>[d]^-{\phi_h|_{\K^\alpha}}
\ar[dl]_-j\\
W\ar[r]_-{\phi_h}&{\K^\alpha.}}$$
Again, by \cite[Lemma (3.1)]{Granas}, $\phi_h|_{\K^\alpha}$ is a Lefschetz map and
\be\label{lef2}\Lambda(\phi_h|_{\K\alpha})=\Lambda(\bar\phi_h).\ee
\indent We now show that $\phi_h|_{\K\alpha}$ is homotopic to the identity $\id:\K^\alpha\to\K^\alpha$. To this end, let us define $\Phi:\K^\alpha\times [0,1]\to\K^\alpha$ by
$$\Phi(x,t) :=
\begin{cases}
\phi_{th}(x), &\text{for}\; x\in K^\alpha,\; t \in (0,1],\\
x , &\text{for}\; x\in K^\alpha,\; t = 0.
\end{cases}$$
By Lemma \ref{lem:0.2} (jv), $\Phi$ is continuous on $\K^\alpha\times (0,1]$. Let $x_n\to x_0$ in $\K^\alpha$ and $t_n\searrow 0$. Then
$$\|\Phi(x_n,t_n)-\Phi(x_0,0)\|_\alpha=\|\phi_{t_nh}(x_n) - x_0\|_\alpha\leq S_1(n)+S_2(n),$$
where
$$S_1(n)=\|J_{t_n h} \circ r (x_n + h_n \F(x_n)) - J_{t_n h}(x_n + t_nh \F(x_n))\|_\alpha,\;\; S_2(n)=\| J_{t_n h}(x_n + t_nh \F(x_n)) - x_0 \|_\alpha.$$
By Lemma \ref{lem:0.2} (j), there is $C_\alpha >0$ such that
\begin{align*}
S_1(n)&\leq  \frac{C_\alpha}{(t_n h)^\alpha ( 1- t_nh\omega)^{1-\alpha}}\|r(x_n +t_nh\F(x_n)) - (x_n + t_nh\F(x_n)\|\\
&\leq \frac{C_\alpha}{(t_n h)^\alpha ( 1- t_nh\omega)^{1-\alpha}}Ld(x_n + t_nh\F(x_n),\K)
\leq \frac{C_\alpha L}{(t_n h)^\alpha ( 1- t_nh\omega)^{1-\alpha}} t_nh\|\F(x_n)\|\\
&\leq \frac{C_\alpha L}{( 1- t_nh\omega)^{1-\alpha}}(t_n h)^{1-\alpha} \|\F(x_n)\|\to 0,\quad n\to \infty,
\end{align*}
since the sequence $(\|\F(x_n)\|)$ is bounded. To estimate $S_2(n)$ note that
$$S_2(n) \leq \|J_{t_n h}(x_n + t_nh \F(x_n)) - J_{t_n h}(x_n)\|_\alpha + \|J_{t_n h}(x_n) - x_0\|_\alpha.$$
The first summand satisfies
\[
\|J_{t_n h}(x_n + t_nh \F(x_n)) - J_{t_n h}(x_n)\|_\alpha \leq \frac{C_\alpha}{(1-t_n h\omega)^{1-\alpha}}(t_n h)^{1-\alpha}\|\F(x_n)\|\to 0, \quad n\to\infty,
\]
and, as in \eqref{d}, $\|J_{t_nh}x_n - x_0\|_\alpha\to 0$ as $n\to\infty$.\\
\indent Now it is clear that
\be\label{lef3}\Lambda(\phi_h|_{\K\alpha})=\Lambda(\id)=\lambda(\id)=
\chi(\K^\alpha)=\chi(\K).\ee In view of Remark \ref{uwagi} (3), \eqref{lef1}, \eqref{lef2} and \eqref{lef3} we conclude the proof.
\end{proof}
\begin{remark}\label{ostatek} {\em If assumptions $(D_1)$ -- $(D_6)$ are satisfied and $\F$ is defined on $\K^\alpha$, then $\deg_\K(\A,\F;U)$ does not depend on open $U\subset X^\alpha$ such that $\K^\alpha\subset U$. Therefore we may suppress it from the notation and write $\deg_\K(\A,\F)$. The normalization property implies in particular that $\deg_\K(\A,0)=\chi(\K)$.
}\end{remark}
The following corollary is convenient to use in applications.
\begin{corollary}\label{wniosek kon} Assume that $\K\subset X$ is closed convex and bounded, $\F:\K^\alpha\to X$ is continuous and tangent, namely $\F(u)\in T_\K(u)$, for every $u\in\K^\alpha$, and let $C:=\{u\in \K^\alpha\mid \A u=t\F(u)\; \text{for some}\; t\in [0,1]\}$. If $\F(C)$ is bounded, then there is $u_0\in\K\cap D(\A)$ such that $\A u_0=\F(u_0)$.
\end{corollary}
\begin{proof} The assumptions imply that a map $\H(u,t):\K^\alpha\to X$ given by $\H(u,t):=t\F(u)$ for $u\in\K^\alpha$, provide an admissible homotopy joining $\F$ to the constant map 0. Then the assertion follows as an immediate consequence of the existence, homotopy and normalization properties of the degree.\end{proof}
\begin{remark}\emph{
If the operator $\A$ is $m$-accretive, then Corollary \ref{wniosek kon} is true if $\K$ is not bounded, but $C$ is bounded. It is an open problem to get this result without a provisional assumption of $m$-accretivity.
}\end{remark}

\noindent\textbf{Acknowledgement.} The first author wishes to express his gratitude towards Research Centre Erwin Schr\"odinger International Institute for Mathematics and Physics
(ESI) of the University of Vienna  for support during his stay at the ESI. The second author  was partially supported by the grant 346300 for IMPAN from the Simons Foundation and the matching 2015-2019 Polish MNiSW fund. The authors are grateful to Aleksander \'{C}wiszewski for fruitful discussions.


\begin{bibdiv}
	\begin{biblist}		

\bib{Adams}{book}{
   author={Adams, Robert A.},
   title={Sobolev spaces},
   note={Pure and Applied Mathematics, Vol. 65},
   publisher={Academic Press [A subsidiary of Harcourt Brace Jovanovich,
   Publishers], New York-London},
   date={1975},
   pages={xviii+268},
   review={\MR{0450957}},
}	

\bib{Alikakos}{article}{
author={Alikakos, N.},
title={Remarks on invariance in reaction-diffusion equations},
journal={Nonl. Anal. Theory, Meth. \& Appl.},
volume={5},
date={1981},
pages={593--614}
}

\bib{Amann}{article}{
author={Amann, H.},
title={Invariant sets and existence theorems fo semilinear parabolic and elliptic systems},
journal={J. Math. Anal Appl.},
volume={65},
date={1978},
pages={432--467}	
}



	
\bib{TerElst}{article}{
   author={Arendt, Wolfgang},
   author={ter Elst, A. F. M.},
   title={From forms to semigroups},
   conference={
      title={Spectral theory, mathematical system theory, evolution
      equations, differential and difference equations},
   },
   book={
      series={Oper. Theory Adv. Appl.},
      volume={221},
      publisher={Birkh\"{a}user/Springer Basel AG, Basel},
   },
   date={2012},
   pages={47--69},
   review={\MR{2953980}},
}	

\bib{ArendtKreuter}{article}{
   author={Arendt, Wolfgang},
   author={Kreuter, Marcel},
   title={Mapping theorems for Sobolev spaces of vector-valued functions},
   journal={Studia Math.},
   volume={240},
   date={2018},
   number={3},
   pages={275--299},
   issn={0039-3223},
   review={\MR{3731026}},
   doi={10.4064/sm8757-4-2017},
}

\bib{Aubin}{book}{
   author={Aubin, Jean-Pierre},
   author={Frankowska, H\'{e}l\`ene},
   title={Set-valued analysis},
   series={Modern Birkh\"{a}user Classics},
   note={Reprint of the 1990 edition [MR1048347]},
   publisher={Birkh\"{a}user Boston, Inc., Boston, MA},
   date={2009},
   pages={xx+461},
   isbn={978-0-8176-4847-3},
   review={\MR{2458436}},
   doi={10.1007/978-0-8176-4848-0},
}

\bib{Beber}{article}{
   author={Bebernes, J. W.},
   author={Schmitt, K.},
   title={Invariant sets and the Hukuhara-Kneser property for systems of
   parabolic partial differential equations},
   journal={Rocky Mountain J. Math.},
   volume={7},
   date={1977},
   number={3},
   pages={557--567},
   issn={0035-7596},
   review={\MR{600519}},
   doi={10.1216/RMJ-1977-7-3-557},
}




\bib{BenK}{article}{
author={Ben-El-Mechaiekh, H.},
author={Kryszewski, W.},
title={Equilibria of set-valued maps on nonconvex domains},
journal={Transaction of Amer. Math. Soc.},
volume={349},
date={1997},
number={10},
pages={4159--4179}
}

\bib{Cann}{article}{
author={Cannarsa, P.},
author={Da Prato, G.},
author={Frankowska, H.},
title={Invariance for quasi-dissipative systems in Banach spaces},
journal={J. Math. Anal. Appl.},
volume={457},
date={2018},
pages={1173--1187}
}

\bib{Carja}{book}{
author={C\^arj\v a, O.}
author={Necula, M.},
author={Vrabie, I.},
title={Viability, invariance and application},
publisher={Elsevier Science B. V. Amsterdam},
date={2007}
}

\bib{Chueh}{article}{
author={Chueh, K. N.},
author={Conley, C. C.},
author={Smoller, J. A.},
title={Positively invariant regions for systems of nonliear reaction-diffusion equation},
journal={Indiana Univ. Math. J.},
volume={26},
date={1977},
pages={373--391}
}

\bib{Cholewa}{book}{
author={Cholewa, J. W.},
author={Dlotko, T.},
title={Global Attractors in Abstract Parabolic Equations},
series={London Mathematical Society Lecture Note Series},
volume={278},
publisher={Cambridge University Press},
date={2000},
pages={xii+235},
isbn={978-0-521-79424-4}
}

\bib{Cwiszewski}{article}{
author={\'{C}wiszewski, A.},
title= {Topological degree methods
for perturbations of operators generating compact $C_0$
semigroups},
journal={J. Differential Equations},
date={2006},
volume={220},
number={2},
pages={434-477}
}



\bib{CwiszewskiKryszewski2}{article}{
   author={\'{C}wiszewski, Aleksander},
   author={Kryszewski, Wojciech},
   title={Homotopy invariants for tangent vector fields on closed sets},
   journal={Nonlinear Anal.},
   volume={65},
   date={2006},
   number={1},
   pages={175--209},
   issn={0362-546X},
   review={\MR{2226264}},
   doi={10.1016/j.na.2005.09.010},
}



\bib{CwiszewskiKryszewski}{article}{
   author={\'{C}wiszewski, Aleksander},
   author={Kryszewski, Wojciech},
   title={Constrained topological degree and positive solutions of fully
   nonlinear boundary value problems},
   journal={J. Differential Equations},
   volume={247},
   date={2009},
   number={8},
   pages={2235--2269},
   issn={0022-0396},
   review={\MR{2561277}},
   doi={10.1016/j.jde.2009.06.025},
}
	
\bib{Egert}{article}{
author={Egert,Moritz},
author={Tolksdorf, Patrick},
title={Characterization of Sobolev functions that vanish on a part of the boundary},
journal={Disc. Cont Dynamical Sytems},
volume={10},
date={2017},
number={4},
pages={729--743},
doi={doi:10.3934/dcdss.2017037},


}	
		
\bib{EngelNagel}{book}{
   author={Engel, Klaus-Jochen},
   author={Nagel, Rainer},
   title={One-parameter semigroups for linear evolution equations},
   series={Graduate Texts in Mathematics},
   volume={194},
   note={With contributions by S. Brendle, M. Campiti, T. Hahn, G. Metafune,
   G. Nickel, D. Pallara, C. Perazzoli, A. Rhandi, S. Romanelli and R.
   Schnaubelt},
   publisher={Springer-Verlag, New York},
   date={2000},
   pages={xxii+586},
   isbn={0-387-98463-1},
   review={\MR{1721989}},
}	

\bib{Gilbarg}{book}{
   author={Gilbarg, David},
   author={Trudinger, Neil S.},
   title={Elliptic partial differential equations of second order},
   series={Classics in Mathematics},
   note={Reprint of the 1998 edition},
   publisher={Springer-Verlag, Berlin},
   date={2001},
   pages={xiv+517},
   isbn={3-540-41160-7},
   review={\MR{1814364}},
}	
		
\bib{Giaquinta}{book}{
   author={Giaquinta, Mariano},
   author={Martinazzi, Luca},
   title={An introduction to the regularity theory for elliptic systems,
   harmonic maps and minimal graphs},
   series={Appunti. Scuola Normale Superiore di Pisa (Nuova Serie) [Lecture
   Notes. Scuola Normale Superiore di Pisa (New Series)]},
   volume={11},
   edition={2},
   publisher={Edizioni della Normale, Pisa},
   date={2012},
   pages={xiv+366},
   isbn={978-88-7642-442-7},
   isbn={978-88-7642-443-4},
   review={\MR{3099262}},
   doi={10.1007/978-88-7642-443-4},
}			
	


\bib{Granas}{article}{
   author={Granas, Andrzej},
   title={The Leray-Schauder index and the fixed point theory for arbitrary
   ANRs},
   journal={Bull. Soc. Math. France},
   volume={100},
   date={1972},
   pages={209--228},
   issn={0037-9484},
   review={\MR{0309102}},
}
			
		
\bib{GranasDugundji}{book}{
   author={Granas, Andrzej},
   author={Dugundji, James},
   title={Fixed point theory},
   series={Springer Monographs in Mathematics},
   publisher={Springer-Verlag, New York},
   date={2003},
   pages={xvi+690},
   isbn={0-387-00173-5},
   review={\MR{1987179}},
   doi={10.1007/978-0-387-21593-8},
}		

\bib{Guesmia}{article}{
author = {Guesmia, S.},
journal ={Differential and Integral Equations},
volume={23},
number = {11/12},
pages = {1091--1103},
title = {Garding inequality on unbounded domains},
year = {2010}
}

\bib{Haase}{book}{
author={Haase, M.},
title={The Functional Calculus for Sectorial Operators},
series={Operator Theory: Advances and Applications}
volume={169},
publisher={Birkhäuser, Basel},
date={2006},
pages={xiv+392},
isbn={�3-7643-7515-9}
}
		
\bib{Henry}{book}{
	author={Henry, Daniel},
   	title={Geometric theory of semilinear parabolic equations},
   	series={Lecture Notes in Mathematics},
   	volume={840},
   	publisher={Springer-Verlag, Berlin-New York},
   	date={1981},
   	pages={iv+348},
   	isbn={3-540-10557-3},
   	review={\MR{610244}},
}	
	
\bib{Karsatos}{book}{
author={Kartsatos, A. G},
title={Recent results involving
compact perturbations and compact resolvents of accretive
operators in Banach spaces},
series={Proceedings of the First World
Congress of Nonlinear Analysts, Tampa, Florida},
date={1992},
publisher={Walter de
Gruyter, New York, 1995}
}

\bib{KryszewskiSiemianowski}{article}{
   author={Kryszewski, Wojciech},
   author={Siemianowski, Jakub},
   title={The Bolzano mean-value theorem and partial differential equations},
   journal={J. Math. Anal. Appl.},
   volume={457},
   date={2018},
   number={2},
   pages={1452--1477},
   issn={0022-247X},
   review={\MR{3705363}},
   doi={10.1016/j.jmaa.2017.01.040},
}

\bib{Kunst}{book}{
   author={Kunstmann, Peer},
   author={Weis, Lutz},
   title={Maximal $L^p$-regularity for Parabolic Equations, Fourier
Multiplier Theorems and $H^\infty$-functional Calculus},
   series={Lecture Notes in Mathematics},
   volume={1885},
   editor={J.-M. Morel},
   editor={F. Takens},
   editor={B. Teissier},
   publisher={Springer-Verlag Berlin, Heidelber, New York},
   date={2004},
   pages={65--312},
   isbn={3-540-23030-0},
}	
\bib{Kuiper1}{article}{
author={Kuiper, H. J.},
title={Invariant set for nonlinear elliptic and parabolic systems},
journal={SIAM J. Math. Anal.},
volume={11},
date={1980},
pages={1075--1103}
}




\bib{Kuiper2}{article}{
author={Kuiper, H. J.},
title={Positively invariant regions for strongly coupled reaction-diffusion systems with balance law},
journal={J. Math. Anal. Appl.},
vplume={249},
date={2000},
pages={340--350}
}



\bib{S}{article}{
author={Le, Vy Khoi},
author={Schmitt, Klaus},
title={Some general concepts of sub and supersolutions for nonlinear elliptic problems},
journal={Topol. Methods. Nonlinear Anal.},
volume={28},
date={2006},
number={1},
pages={87--103},
review={MR2262257}
}
\bib{Marc}{article}{
author={Marciniak-Czochra, A.},
author={Kimmel, M.},
title={Modelling of early lung cancer
progression: influence of growth factor production and cooperation between
partially transformed cells},
journal={Math. Mod. Meth. Appl. Sci.},
volume={17},
date={2007},
pages={1693--1719}
}
\bib{MarcusMizel}{article}{
   author={Marcus, Moshe},
   author={Mizel, Victor J.},
   title={Every superposition operator mapping one Sobolev space into
   another is continuous},
   journal={J. Funct. Anal.},
   volume={33},
   date={1979},
   number={2},
   pages={217--229},
   issn={0022-1236},
   review={\MR{546508}},
   doi={10.1016/0022-1236(79)90113-7},
}



\bib{MarcusMizel1}{article}{
   author={Marcus, Moshe},
   author={Mizel, Victor J.},
   title={Nemitsky operators on Sobolev spaces},
   journal={Arch. Rational Mech. Anal.},
   volume={51},
   date={1973},
   pages={347--370},
}
\bib{MarcusMizel2}{article}{
author={Marcus, Moshe},
author={Mizel, Victor J.},
   title={Complete Characterization of Functions Which Act, Via
   Superposition, on Sobolev Spaces},
journal={ Trans. Amer. Math. Soc.},
volume={251},
date={1979},
pages={187-218},
}

\bib{Martin}{book}{
author={Martin, Jr. R. H.},
title={Nonlinear operators and differential equations in banach spaces},
publisher={Wiley-Interscience, New York},
date={1976}
}

\bib{McLean}{book}{
author={McLean, W>},
title={Strongly Elliptic Systems and Boundary Integral Equations},
publisher={Cambridge University Press},
date={2000},
pages={xiv+357},
isbn={0 521 66332 6}
}

\bib{McKenna}{article}{
author={McKenna, P.},
author={Walter, W.},
title={On the Dirichlet problem for
elliptic systems},
journal={Applicable Analysis},
volume={21},
date={1986},
number={3},
pages={207--224}


}



\bib{Miti}{article}{
author={Mitidieri, Enzo},
author={Sweers, Guido},
title={Existence of a maximal solution for quasimonotone elliptic systems},
journal={Differential Integral Eq.},
volume={7},
date={1993},
number={6},
pages={1495--1510}
}

\bib{Muller}{article}{
author={M\"uller, Max},
title={\"Uber das Fundamentaltheorem in der Theorie der gew\"ohnlichen Differentialgleichungen},
journal={Math. Z.}
volume={26},
date={1927},
number={1},
pages={619--645}
}

\bib{Muller-book}{book}{
author={M\"uller, Max},
title={\"Uber die Eindeutigkeit der Integrale eines Systems gew\"ohnlicher Differentialgleichungen und die Konvergenz einer Gattung von Verfahren zur Approximation dieser Integrale},
publisher={Walter de Gruyter, Heidelberg},
date={1927}
}
\bib{Necula}{article}{
author={Necula, M.},
author={Popescu, M.},
author={Vrabie, I.},
title={Viability for differential inclusions on graphs},
journal={Set-Valued Anal.}
volume={16}
date={2008},
pages={961--981}
}


\bib{Pazy}{book}{
   author={Pazy, A.},
   title={Semigroups of linear operators and applications to partial
   differential equations},
   series={Applied Mathematical Sciences},
   volume={44},
   publisher={Springer-Verlag, New York},
   date={1983},
   pages={viii+279},
   isbn={0-387-90845-5},
   review={\MR{710486}},
   doi={10.1007/978-1-4612-5561-1},
}	
\bib{Plum}{article}{
   author={Plum, Michael},
   title={Shape-invariant bounds for reaction-diffusion systems with unequal
   diffusion coefficients},
   journal={J. Differential Equations},
   volume={73},
   date={1988},
   number={1},
   pages={82--103},
   issn={0022-0396},
   review={\MR{938216}},
   doi={10.1016/0022-0396(88)90119-2},
}




\bib{Redheffer}{article}{
author={Redheffer, R.},
author={Walter, W.},
title={Invariant sets for systems of partial differential equations},
journal={Arch. Rat. Mech. Anal.},
volume={67},
date={1978},
pages={41--52}
}


\bib{Redlinger}{article}{
author={Redlinger, R.},
title={Invariant sets for strongly coupled reaction-diffusion systems under general boundary conditions},
journal={Arch. Rat. Mech. Anal.},
volume={108},
date={1989},
pages={281--291}
}	

\bib{schroder}{article}{
   author={Schr\"{o}der, Johann},
   title={Shape-invariant bounds and more general estimates for
   vector-valued elliptic-parabolic problems},
   journal={J. Differential Equations},
   volume={45},
   date={1982},
   number={3},
   pages={431--460},
   issn={0022-0396},
   review={\MR{672717}},
   doi={10.1016/0022-0396(82)90037-7},
}


\bib{Show}{book}{
   author={Showalter, R. E.},
   title={Monotone operators in Banach space and nonlinear partial
   differential equations},
   series={Mathematical Surveys and Monographs},
   volume={49},
   publisher={American Mathematical Society, Providence, RI},
   date={1997},
   pages={xiv+278},
   isbn={0-8218-0500-2},
   review={\MR{1422252}},
}

\bib{Smoller}{book}{
   author={Smoller, Joel},
   title={Shock waves and reaction-diffusion equations},
   series={Grundlehren der Mathematischen Wissenschaften [Fundamental
   Principles of Mathematical Science]},
   volume={258},
   publisher={Springer-Verlag, New York-Berlin},
   date={1983},
   pages={xxi+581},
   isbn={0-387-90752-1},
   review={\MR{688146}},
}

	
\bib{Walter-survey}{article}{
author={Walter, W.},
title={Differential inequalities and maximum principles: theory, new methods and applications},
journal={Nonlinear Anal. Theory, Meth. \& Appl.},
volume={30},
date={1997},
number={8},
pages={4695--4711}
}
\bib{Walter-book}{book}{
author={Walter, W.},
title={Differential and Integral Inequalities},
series={Ergebnisse der Mathematik und ihrer Grenzgebiete}
volume={Band 55},
publisher={Springer-Verla, Heidelberg} ,
date={1970}
}

\bib{Wein}{article}{
author={Weinberger, H. F.},
title={Invariant sets for weakly coupled parabolic and elliptic systems},
journal={Rend. Mat.},
volume={8},
date={1975},
pages={295--310}
}


\bib{Yagi}{book}{
author={Yagi, A.},
title={Abstract Parabolic Evolution Equations and their Applications},
series={Springer Monographs in Mathematics},
publisher={Springer-Verlag, Heidelberg},
date={2010},
pages={xviii+581},
isbn={978-3-642-04630-8}
}			

			\end{biblist}
	\end{bibdiv}
\end{document}